\documentclass[10pt%,reqno
]{amsart}
\usepackage{graphicx, amssymb, amsmath, amsthm}
\numberwithin{equation}{section}
\usepackage[numbers]{natbib}
\usepackage{color}
\usepackage{amscd,hyperref}
\hypersetup{hidelinks}
\usepackage{geometry}
\usepackage{comment}
\usepackage{tikz}
\newcommand{\pa}{\partial}
\newtheorem{example}{Example}[section]
\newcommand{\norm}[1]{\left\|#1\right\|}
\newcommand{\Ree}{\operatorname{Re}}

\let\Re=\undefined\DeclareMathOperator*{\Re}{Re}
\let\Im=\undefined\DeclareMathOperator*{\Im}{Im}

\newcommand{\R}{\mathbb{R}}
\newcommand{\C}{\mathbb{C}}

\newcommand{\ip}[2]{\left\langle #1,#2\right\rangle}

\newcommand{\Z}{\mathbb{Z}}

\newtheorem{assumption}{Assumption}[section]

\newtheorem{theorem}{Theorem}[section]

\newtheorem{lemma}[theorem]{Lemma}

\newtheorem{corollary}[theorem]{Corollary}

\newtheorem{proposition}[theorem]{Proposition}
\def\Id{\mathop{\mathrm{Id}}\nolimits}

\theoremstyle{definition}
\newtheorem{definition}[theorem]{Definition}
\newtheorem{remark}[theorem]{Remark}
\def\<{{\langle}}
\def\>{{\rangle}} 
\def\supp{\mathop{\mathrm{supp}}\nolimits}
\def\curl{\mathop{\mathrm{curl}}\nolimits}
\def\div{\mathop{\mathrm{div}}\nolimits}

\newcommand{\Fcurv}{\mathcal{F}}

\makeatletter  
\newcommand{\Extend}[5]{\ext@arrow0099{\arrowfill@#1#2#3}{#4}{#5}}
\makeatother

\begin{document}
	\title[Defocusing damped NLS ]{Scattering theory for 3d cubic damped  magnetic Schr\"odinger equation }

	\author[L. Fanelli]{Luca Fanelli}
	\address{Luca Fanelli \newline
		Ikerbasque, Basque Foundation for Science; Departamento de Matem\'aticas, Euskal Herriko Unibertsitatea(EHU), 48940 Leioa, Spain; BCAM - Basque Center for Applied Mathematics, 48009 Bilbao, Spain}
	\email{lfanelli@bcamath.org}
	\author[H. Mizutani]{Haruya Mizutani}
	\address{Haruya Mizutani\newline
		\indent Department of Mathematics, Graduate School of Science, The University of Osaka, Toyon-
		aka, Osaka 560-0043, Japan}
	\email{haruya@math.sci.osaka-u.ac.jp}
	\author[Y. Song]{Yilin Song}
	\address{Yilin Song
		\newline \indent The Graduate School of China Academy of Engineering Physics,
		Beijing 100088,\ P. R. China}
	\email{songyilin21@gscaep.ac.cn}
	\author[Y. Wang]{Ying Wang}
	\address{Ying Wang\newline
		BCAM - Basque Center for Applied Mathematics, 48009, Bilbao, Spain}
	\email{ywang@bcamath.org}
	
	\author[J. Zheng]{Jiqiang Zheng}
	\address{Jiqiang Zheng
		\newline Institute of Applied Physics and Computational Mathematics, Beijing, 100088, China.
		\newline National Key Laboratory of Computational Physics, Beijing, 100088, China}
	\email{zheng\_jiqiang@iapcm.ac.cn, zhengjiqiang@gmail.com}

	\begin{abstract}
We consider the three-dimensional defocusing cubic nonlinear Schr\"odinger equation with variable coefficients, a magnetic potential, and a non-negative localized damping term,
\[
i\partial_tu+(\nabla-iA)\cdot G(\nabla-iA)u+ia(x)u=|u|^2u,
\qquad t>0,\quad x\in\mathbb R^3.
\]
No non-trapping condition is imposed on the metric $G$. Instead, the variable-coefficient region is assumed to be contained in the effective damping region. Under a one-centre condition on the tangential magnetic field, we prove global well-posedness for initial data in $H^{1+\varepsilon}$, uniform mass and energy bounds, and show the local energy decay. To obtain scattering, we impose a support condition on  the full magnetic field inside the damping region. Under these stronger assumptions, the solution scatters to a free Schr\"odinger evolution in $H^s$ for every $0\le s<1$. The appendix discusses a separate constant-damping framework for abstract Hamiltonians.
\end{abstract}

	\maketitle
	
	\begin{center}
		\begin{minipage}{100mm}
		\end{minipage}
	\end{center}
	
	% \tableofcontents %
	
	\section{Introduction}
	
	\noindent
	In this paper, we study the Cauchy problem for the three-dimensional defocusing cubic nonlinear Schr\"odinger equation with variable coefficients, magnetic potentials, and a non-negative damping term
	\begin{equation}\label{equ:DMNLS}
		\begin{cases}
			i\partial_tu+\Delta_{G}u+iau=|u|^{2}u,\quad (t,x)\in\R^+\times\R^3,\\
			u(0,x)=u_0(x),
		\end{cases}
	\end{equation}
	where $u(t,x):\R^+\times\R^3\to\C$ and the operator $\Delta_G$ is defined by
	$$\Delta_G=\nabla_A\cdot\big(G\nabla_A)=\sum_{j,k=1}^3\big(\partial_j-iA_j(x)\big)\big(G_{j,k}(x)\big(\pa_k-iA_k(x)\big)\big)$$
	with $\nabla_A=\nabla-iA(x)$, $A=(A_j)_{j=1}^3\in C^\infty(\R^3,\R^3)$ and $G=(G_{j,k})_{j,k=1}^3\in C^\infty(\R^3,\R^{3\times3})$ being symmetric and uniformly positive-definite: there exist $c,C>0$ such that for any $x,\xi\in \R^3$, 
	\begin{equation}\label{equ:Gpd}
		C|\xi|^2\geq G(x)\xi\cdot\xi\geq c|\xi|^2. 
	\end{equation}
	Moreover, we assume that the metric is the compact perturbation of Euclidean metric 
	$$G-I\in C_c^\infty(\R^3, \R^{3\times3}).$$
	The damping potential $a(x)\in C_c^\infty(\R^3)$ is assumed to be non-negative.

	Nonlinear damped Schr\"odinger equations arise in several physical models, including nonlinear optics, plasma physics, and fluid dynamics; see \cite{Fib01,GRH80,Tsu84}.
	
	When $G=I$, $A=0$ and $a=0$,  Cauchy problem \eqref{equ:DMNLS}  becomes the classical defocusing cubic nonlinear Schr\"odinger equation
	\begin{equation}\label{equ:NLS3dc}
		\begin{cases}
			i\pa_tu+\Delta u=|u|^{2}u,\quad (t,x)\in\R\times\R^3,\\
			u(0,x)=u_0(x)\in H^1(\R^3). 
		\end{cases}
	\end{equation}
	The global well-posedness of \eqref{equ:NLS3dc} in the energy space $H^1(\R^3)$ easily follows from Strichartz estimates and the fixed point argument, see  Cazenave \cite{Caz}. A global solution $u$ is said to  scatter if there exist $u_\pm\in H^1(\R^3)$ such that	$$\lim_{t\to\pm\infty}\|u(t)-e^{it\Delta}u_\pm\|_{H^1(\R^3)}=0.$$
	Ginibre and Velo \cite{GV85} proved the scattering result for \eqref{equ:NLS3dc} by exploiting the almost finite propagation speed	$$\int_{|x|\geq a} |u(t,x)|^2 \, dx \leq \int \min\left(\frac{|x|}{a},
	1\right) |u(t_0)|^2 \, dx + \frac{C}{a} \cdot |t-t_0|$$ for large
	spatial scale and the Morawetz inequality
	\begin{equation}\label{cmrsz}
		\iint_{\R\times \R^3}\frac{|u(t,x)|^{4}}{|x|} \, dtdx \lesssim
		\|u\|_{L_t^\infty\dot H^\frac12}^2
	\end{equation} for small spatial scale. Lately,  in \cite{CKSTT04}, Colliander, Keel,  Staffilani,  Takaoka, and  Tao
	establish the interaction Morawetz estimate
	\begin{equation}\label{equ:intMorIteam}
		\|u\|_{L_{t,x}^4(\R\times\R^3)}\leq C\|u_0\|_{L_x^2(\R^3)}^\frac34\sup_{t\in\R}\|u(t)\|_{\dot{H}^1_x(\R^3)}^\frac14,
	\end{equation}
	and give another simple proof of scattering result of \cite{GV85}. 
	\subsection{The trapping problem and the localized-damping mechanism}
For a non-flat metric $G$ and a non-zero magnetic potential $A$, Mizutani \cite{Miz13} established local-in-time Strichartz estimates with a loss of derivatives. Under a non-trapping condition this loss can be removed, and Cassano--D'Ancona \cite{CD16} proved scattering for the variable-coefficient magnetic Schr\"odinger equation in the energy space. We also refer to \cite{Fanelli2,Cacciafesta} for smoothing estimates for related variable-coefficient magnetic dispersive models.

The present paper addresses a different regime. We allow the metric to have trapped geodesics, so global-in-time Strichartz estimates without loss are not available in general. Trapping may support non-dispersive components and obstruct scattering. Our aim is to show that a localized damping term can remove the energy remaining in the trapping region, while the radiative component escaping to infinity is asymptotically governed by the free Schr\"odinger flow. The magnetic terms require two distinct controls: a one-centre projected-field condition for the radial virial estimate, and a full-curvature condition for the interaction Morawetz argument.

	Define the mass and energy associated with \eqref{equ:DMNLS} by
	\begin{align}\label{equ:massG}
		M(u(t))&:=\int_{\R^3}|u(t,x)|^2\,dx,\\\label{equ:energyG}
		E(u(t))&:=\frac12\int_{\R^3} G\nabla_A u\cdot \overline{\nabla_Au}\,dx+\frac14\int_{\R^3}|u(t,x)|^4\,dx. 
	\end{align}
	If the damping term $a\equiv0$, they are conserved by the flow of \eqref{equ:DMNLS}, while this is not the case if $a\not\equiv0$. 	
	
	Recently, Lafontaine and Shakarov \cite{LS25} established global well-posedness and scattering for the three-dimensional defocusing cubic Schr\"odinger equation  under the control condition
	\begin{equation}\label{equ:assumpona}
		\operatorname{supp}(G-I)\Subset\{x\in\R^3:a(x)>0\}.
	\end{equation}
	The local energy decay plays the key role in their proof. For the linear case with $A=0$ and $a=0$, we refer to \cite{Burq-Bouclet} for details. %It can be proven by establishing the weighted $L^2$ resolvent estimate.
	
	Inspired by \cite{LS25}, we extend their analysis to the magnetic setting. We write the magnetic potential as the one-form
	\[
	A=A_1\,dx_1+A_2\,dx_2+A_3\,dx_3
	\]
	and denote its full spatial magnetic curvature by
	\begin{equation}\label{def:full-curvature}
		\Fcurv=dA,
		\qquad
		\Fcurv_{jk}:=\partial_jA_k-\partial_kA_j,
		\qquad 1\le j,k\le3.
	\end{equation}
	In dimension three, $\Fcurv$ is equivalent, through the Hodge-star identification, to the full magnetic field
	\[
	B:=\curl A=(\partial_2A_3-\partial_3A_2,\partial_3A_1-\partial_1A_3,\partial_1A_2-\partial_2A_1).
	\]
	The centered tangential component is
	\[
	B_{\tau,0}(x):=\frac{x}{|x|}\wedge\curl A(x).
	\]
	We occasionally write $B_\tau=B_{\tau,0}$. This projected field is adapted to radial multipliers centred at the origin, whereas the interaction Morawetz identity involves the full curvature $\Fcurv$.

	\begin{assumption}[The tangential magnetic control]\label{assum:Assumpcond}
		Assume that $A\in C^\infty(\R^3;\R^3)$, that $\partial_x^\alpha A\in L^\infty(\R^3;\R^3)$ for every multi-index $\alpha$, and that the Coulomb gauge condition
		\[
		\div A=0
		\]
		holds. In addition, assume that one of the following conditions is satisfied:
		\begin{enumerate}
			\item[(i)] the one-centre smallness condition:  
			\begin{equation}\label{equ:Btaucond}
				\big\||x|^{3/2}B_{\tau,0}\big\|_{L_r^2L^\infty(S_r)}^2
				:=\int_0^\infty r^3\sup_{|x|=r}|B_{\tau,0}(x)|^2\,dr<\frac{2c_0}{(2c_0+1)^2},
			\end{equation}
            for some fixed constant $c_0>0$.
			\item[(ii)] the one-centre support condition
			\begin{equation}\label{control_B}
				\operatorname{supp}B_{\tau,0}\Subset\{x\in\R^3:a(x)>0\}.
			\end{equation}
		\end{enumerate}
	\end{assumption}

	\paragraph{Role of the one-centre hypothesis.}
	Assumption~\ref{assum:Assumpcond} is used only in the proof of the uniform energy bound and local energy decay in which the virial argument is invoked. However , to obtain the scattering in $H^s$, the interaction Morawetz estimate plays the central role. The interaction Morawetz identity  involves the magnetic field with all directions instead of the tangential direction.
	Therefore, we need to impose the following stronger conditions.
	\begin{assumption}[A stronger condition to yield scattering]\label{assum:scattering}
		Let
		\[
		C_j:=\{x\in\R^3:2^j\le |x|\le2^{j+1}\},\qquad j\in\Z,
		\]
		and assume that 
		\begin{equation}\label{condition1}
			\sum_{j\ge0}2^{3j}\|A\|_{L^\infty(C_j)}^2<\infty.
		\end{equation}
		Moreover, assume that the full magnetic curvature is supported in the effective damping region,
			\begin{equation}\label{control-full-curvature}
				\operatorname{supp}\Fcurv\Subset\{x\in\R^3:a(x)>0\}.
			\end{equation}

	\end{assumption}
\begin{remark}\label{magnetic-condition}
We will see that the condition \eqref{control-full-curvature} implies the following bounds of the magnetic curvature $\Fcurv$. More precisely, 
    	for arbitrary $\beta\in(0,1)$ such that
			\begin{align}
				\mathcal K_1(\Fcurv)&:=\sum_{j\in\Z}2^{j+1}\|\Fcurv\|_{L^\infty(C_j)}^{2-2\beta}<\infty,\label{condition:K1}\\
				\mathcal K_2(\Fcurv)&:=\int_0^\infty r^2\|\Fcurv\|_{L^\infty(S_r)}^{2\beta}\,dr<\infty;\label{condition:K2}
			\end{align}
where $B=\curl A$ is the  magnetic field which enjoys the same $L^\infty$ norm as the magnetic curvature $\Fcurv$. Indeed, since $a\in C_0^\infty(\R^3)$ and $\supp \Fcurv\subset K$ where $K$ is a compactly supported set. Then the two bounds hold naturally.  
\end{remark}
	\begin{remark}[Non-trivial admissible magnetic potentials]\label{rem:magnetic-examples}
		The assumptions above are non-empty. First, let $\psi\in C_c^\infty(\{a>0\})$ be real-valued and non-zero, and set
		\[
		A_\varepsilon(x)=\varepsilon\big(-\partial_2\psi(x),\partial_1\psi(x),0\big).
		\]
		Then $\div A_\varepsilon=0$, $A_\varepsilon$ is smooth and compactly supported, and $\operatorname{supp}(dA_\varepsilon)\Subset\{a>0\}$. Choosing $\psi$ so that $\partial_1^2\psi+\partial_2^2\psi\not\equiv0$ gives a non-zero magnetic field. Hence Assumptions~\ref{assum:Assumpcond} and~\ref{assum:scattering} hold.

		% A non-compactly supported example is
		% \begin{equation}\label{example:decaying-A}
		% A_{\varepsilon,\mu}(x)
		% =\varepsilon(-x_2,x_1,0)(1+|x|^2)^{-5/4}
		% \big[\log(e+|x|^2)\big]^{-\mu},
		% \qquad \mu>\frac12.
		% \end{equation}
		% It is divergence-free and satisfies
		% \[
		% |A_{\varepsilon,\mu}(x)|\lesssim |\varepsilon|\langle x\rangle^{-3/2}
		% [\log(e+|x|^2)]^{-\mu},
		% \qquad
		% |dA_{\varepsilon,\mu}(x)|\lesssim |\varepsilon|\langle x\rangle^{-5/2}
		% [\log(e+|x|^2)]^{-\mu}.
		% \]
		% Consequently, \eqref{condition1} holds for $\mu>1/2$, while \eqref{condition:K1}--\eqref{condition:K2} hold for every $\beta\in(3/5,4/5)$. Since the norm in \eqref{equ:Btaucond} is $O(\varepsilon^2)$, Assumption~\ref{assum:Assumpcond}(i) also holds when $|\varepsilon|$ is sufficiently small.
	\end{remark}

	We separate the global existence and scattering statements because they use different magnetic hypotheses.
	\begin{theorem}[Global well-posedness and local energy decay]\label{T:gwp}
		Let $a\in C_c^\infty(\R^3)$ be non-negative and let $G\in C^\infty(\R^3;\R^{3\times3})$ be symmetric, uniformly positive-definite, and satisfy \eqref{equ:assumpona}. Assume that Assumption~\ref{assum:Assumpcond} holds. Then, for every $\varepsilon>0$ and $u_0\in H^{1+\varepsilon}(\R^3)$, equation~\eqref{equ:DMNLS} has a unique global solution
		\[
		u\in C([0,\infty);H^{1+\varepsilon}(\R^3)).
		\]
		Moreover,
		\[
		M(u(t))\le M(u_0),
		\qquad
		E(u(t))\lesssim E(u_0)+M(u_0),
		\qquad t\ge0,
		\]
		and the integrated local energy estimates of Proposition~\ref{prop:uniformenergy} below hold. In particular, for every $\chi\in C_c^\infty(\R^3)$ and $0\le s<1$,
		\[
		\|\chi u(t)\|_{H^s}\longrightarrow0
		\qquad\text{as }t\to\infty.
		\]
	\end{theorem}

	\begin{theorem}[Scattering]\label{T:main}
		Under the assumptions of Theorem~\ref{T:gwp}, assume in addition Assumption~\ref{assum:scattering}. Then there exists $u_+\in H^1(\R^3)$ such that, for every $0\le s<1$,
		\begin{equation}\label{equ:sca}
			\lim_{t\to+\infty}\|u(t)-e^{it\Delta}u_+\|_{H^s(\R^3)}=0.
		\end{equation}
	\end{theorem}

	\begin{remark}
		The present proof first establishes strong convergence in $L^2$ and then interpolates with the uniform $H^1$ bound. It therefore yields scattering in $H^s$ for $s<1$; the theorem does not assert that $H^1$ scattering is impossible under the stated assumptions.
	\end{remark}
\begin{remark}
  We will provide an abstract scattering result in $H^1$ under a non-trapping assumption later. In contrast, this theorem shows that when trapping is allowed, the present argument yields scattering in $H^s$
for every $s<1$. Obtaining $H^1$-scattering under trapping needs more delicate analysis for the associated linear damped Schr\"odinger flow which is far beyond the scope of this method. This is a widely open problem even if $A=0$. 

\end{remark}

\subsection{A complementary $H^1$ scattering results for non-trapping geometries}
As a complementary result, we will show the $H^1$ scattering for several Hamiltonian $H$. We should emphasize that $H$ is non-trapping in both settings. Before presenting our result, we first recall the classical result of $H^1$ scattering for damped NLS proved by Inui \cite{Inui}. In \cite{Inui}, exponential scattering was established for the classical nonlinear Schr\"odinger equation with constant damping on $\mathbb R^d$:
\begin{align*}
  \begin{cases}
  i\partial_tu+\Delta u+iau=\mu|u|^{p-1}u,&(t,x)\in\R\times\R^d,\\
  u(0,x)=u_0(x).
  \end{cases}
\end{align*}
 More precisely,   for the  defocusing case with $L^2$ subcritical/critical nonlinearity   with  $1<p\leq1+\frac{4}{d}$, he showed that the solution exponentially scatters in $L^2$. In the  focusing case, exponential scattering holds provided $\|u_0\|_{L^2}<\|Q\|_{L^2}$ if $p\leq 1+\frac{4}{d}$. Under certain assumption of the Hamiltonian $H$, the framework of \cite{Inui} can be extended naturally  to the following abstract setting for positive Hamiltonian $H$. Let $\mathcal H:=L^2(\mathbb R^d)$ be a complex Hilbert space with inner product $\ip{\cdot}{\cdot}$. We consider the abstract damped nonlinear Schr\"odinger equation with the form
\begin{equation}\label{eq:damped}
i\partial_t u - Hu - \mu |u|^{p-1}u + i\mathcal B u = 0.
\end{equation}
The case $\mu=+1$ corresponds to the \emph{defocusing} nonlinearity, while $\mu=-1$ corresponds to the \emph{focusing} nonlinearity. 	
	To state our abstract scattering result, we need the following assumptions on the Hamiltonian $H$.  
	\begin{assumption}\label{ass:H}Let $H\ge 0$ be a selfadjoint operator on $\mathcal H$, obtained as the Friedrichs extension
		associated with a densely defined, closed, Hermitian, nonnegative quadratic form $q$ with form domain
		\[
		\mathcal Q(H)=D(H^{1/2}),\qquad q(u)=\norm{H^{1/2}u}_{\mathcal H}^2.
		\]
		We define the norm by
\begin{align*}
  \|u\|_{D(H^\frac12)}:=\|u\|_{\mathcal{H}}+\|H^\frac{1}{2}u\|_{\mathcal{H}}.
\end{align*}
		Moreover, we consider two types of positive Hamiltonian $H$.
		The first one is the classical fractional Laplacian case, i.e. $H=(-\Delta)^s$ while in the second case we assume that the Hamiltonian $H$ is a second-order self-adjoint operator on
		$L^2(\mathbb R^d)$ such that
		\[
		q(u)\simeq \|\nabla u\|_{L^2}^2
		\qquad \text{on } H^1(\mathbb R^d)\cap D(H^{1/2}).
		\]

		Fix $p>1$ and assume we are in a concrete realization $\mathcal H=L^2(\R^d)$.
		Assume moreover that $D(H^{1/2})\hookrightarrow L^{p+1}(\R^d)$ so that
		$\norm{u}_{L^{p+1}}^{p+1}$ is finite for $u\in D(H^{1/2})$. 
	\end{assumption}

	\begin{assumption}
			[Constant damping hypothesis]\label{ass:damp}
			
			Assume that the constant damping $\mathcal{B}=bI$ satisfies $b\geq a$ where $a>0$.
	\end{assumption}

Next, we introduce the notation of \emph{exponentially scattering} for this abstract damped NLS.
	\begin{definition}[Exponential scattering]\label{def:exp-scatt}
				Let $H\ge0$ be a selfadjoint operator on $\mathcal H=L^2(\mathbb R^d;\mathbb C)$ and let
				$\mathcal B$ satisfies Assumption~\ref{ass:damp} with  $a>0$.
				A global solution
				\[
				u \in C([0,\infty);D(H^{1/2}))
				\]
				to the damped nonlinear Schr\"odinger equation
				\[
				i\partial_t u - Hu - \mu|u|^{p-1}u + i\mathcal Bu = 0
				\]
				is said to \emph{exponentially scatter forward in time} if there exists
				$u_+\in D(H^{1/2})$ and $\gamma>0$ such that
				\[
				\lim_{t\to+\infty}
				e^{\gamma t}\,
				\big\|
				u(t)-e^{-bt}e^{itH}u_+
				\big\|_{D(H^{1/2})}
				=0.
				\]
	\end{definition}

Now, we can formulate the abstract exponentially scattering theorem in both focusing and defocusing regime.  We start with the following two priori assumptions, which will be fulfilled by the example provided below. 
	\begin{assumption}[Hamiltonian]\label{ass:Hamiltonian}
		Let $H$ be a positive Hamiltonian and $s\in(0,1]$ Assume that $\gamma_0$ is the maximal index such that the following Sobolev embedding holds for $\gamma\leq\gamma_0:=\frac{2d}{d-2s}$ and $d>2s$,
		\begin{align*}
			D(H^\frac{1}{2})\hookrightarrow L^\gamma,\,\,\gamma\leq\gamma_0.
		\end{align*}
	\end{assumption}
	\paragraph{Gagliardo--Nirenberg and nonlinear estimates.}
		Assume that
		\begin{equation}\label{eq:GN-abstract}
			\|u\|_{L^{p+1}}^{p+1}
			\le C_{GN}\|H^{1/2}u\|_{L^2}^{\frac{d}{2s}(p-1)}
			\|u\|_{L^2}^{(p+1)-\frac{d}{2s}(p-1)}
		\end{equation}
		for $u\in D(H^{1/2})$. We also assume the nonlinear estimate needed in the Strichartz space,
		\begin{equation}\label{eq:abstract-chain-rule}
			\|H^{1/2}(|u|^{p-1}u)\|_{L^{r'}}
			\lesssim \|u\|_{L^r}^{p-1}\|H^{1/2}u\|_{L^r},
		\end{equation}
		for the exponent $r=p+1$, together with the corresponding Lipschitz estimate for differences.

	\begin{remark}\label{rem:abstract-nonlinear-estimates}
		The Gagliardo--Nirenberg inequality and the nonlinear chain rule above are explicit hypotheses of the abstract framework. They do not follow from the Sobolev embedding alone and must be verified for each concrete Hamiltonian.
	\end{remark}
	\begin{remark}

	   We emphasize that if the Riesz transform $H^\frac{1}{2}(-\Delta)^{-\frac12}$ is $W^{s,p}$ bounded for some $s(p)$ and $p>1$, then the above nonlinear estimate holds automatically. Indeed, by using the $W^{s,p}$ boundedness of Riesz transform, one can show the equivalence of Sobolev norm $\|u\|_{D(H^\frac12)}\sim\|u\|_{W^{s,q}}$. Then using the fractional chain rule for $(-\Delta)^\alpha$, one can get this estimate.
	\end{remark}

	\begin{assumption}[Strichartz estimate for the linear flow]\label{ass-strichartz}
		Let $d\geq1$ and $(q,r)\in\mathbb R^2$ such that $2\leq q\leq\infty$ and $2\leq r<\gamma_0$. Then we assume that the following estimate 
		\begin{align*}
			\big\|e^{itH}u_0\big\|_{L_t^qH_x^{s,r}(I\times\mathbb R^d)}\lesssim \|u_0\|_{D(H^\frac{1}{2})},
		\end{align*}
		at least holds for $(q,r)=\big(\frac{4s(p+1)}{d(p-1)},p+1\big)$. We also assume the corresponding retarded inhomogeneous estimate. Here, we define the Sobolev norm $\|f\|_{H_x^{s,r}}:=\|\langle H\rangle^{1/2}f\|_{L^r}$ via functional calculus where $\langle x\rangle:=(1+|x|^2)^\frac12$ 
	\end{assumption}

	Throughout this appendix, $\mathcal B$ denotes the constant damping operator; in the main body, the localized damping is denoted by $a(x)$.
	\begin{theorem}\label{thm-abstract}
		Let $\mu\in\{-1,1\}$ and assume Assumptions~\ref{ass:H}, \ref{ass:damp}, \ref{ass:Hamiltonian}, and \ref{ass-strichartz}, together with \eqref{eq:GN-abstract}--\eqref{eq:abstract-chain-rule}. Then, for every $u_0\in D(H^{1/2})$, equation~\eqref{eq:damped} is globally well posed in the defocusing case $\mu=1$ throughout the energy-subcritical range. In the focusing case $\mu=-1$, assume in addition the standard coercive mass-threshold condition associated with the sharp Gagliardo--Nirenberg inequality, with $1<p\le1+4s/d$. In either case the global solution exponentially scatters forward in time in the sense of Definition~\ref{def:exp-scatt}.
	\end{theorem}
	As a consequence of Theorem \ref{thm-abstract} and the decay of mass, we have the following exponential decay of the energy.
	\begin{corollary}[Exponential decay of the energy]\label{lem:energy-exp-decay}
		Under the assumptions of Theorem~\ref{thm-abstract}, there exists $C=C(u_0)>0$ such that
		\begin{equation}\label{eq:energy-exp-decay}
			|E(u(t))|\le C e^{-2bt},
			\qquad t\ge0.
		\end{equation}
	\end{corollary}

	For the purpose of proving exponential scattering, we restrict ourselves
			to constant damping.

	%\begin{remark}
	%    With some minor change, one can prove the full energy sub-critical regime in the focusing case for small initial data. However, this will not be our main aim, since we focus on the defocusing case.
	%\end{remark}
	\begin{remark}
		
		Similar to \cite{Saanouni}, one can generalize the constant damping operator to time-dependent case $a(t)$. Introducing $\mathbf B(t)=\int_0^tb(s)\,ds$ and a minimal damping $\underline{b}:=\inf\limits_{t>0}\frac{b(t)}{t}$, following a similar strategy, one can prove the exponential scattering in the following sense 
		\begin{align*}
			\lim_{t\to+\infty}\big\|e^{\mathbf B(t)}u(t)-e^{itH}u_+\big\|_{D(H^\frac12)}=0.
		\end{align*}
		If one considers $\mathcal B=b(x)I$, one can  expect the classical scattering holds as shown in \cite{LS25}. 
	\end{remark}

The sharpness of the $s<1$ limitation is illustrated by Theorem \ref{thm-abstract}, where the non-trapping hypothesis is used to ensure the validity of the Strichartz estimate, which yields the full $H^1$-scattering. This dichotomy highlights the genuine obstruction created by trapped geodesics under variable coefficients.

For comparison, we  records the detailed  proof  of  exponential scattering theorem for nonlinear Schr\"odinger equations with constant damping and a general positive Hamiltonian, i.e. Theorem \ref{thm-abstract}. Under global Strichartz estimates and suitable nonlinear estimates, that result yields scattering in the energy space $D(H^{1/2})$, which agrees with $H^1$ for the standard second-order examples considered there. It does not contain Theorems~\ref{T:gwp}--\ref{T:main}: the abstract framework requires global Strichartz estimates without loss of derivatives and therefore excludes the trapping metrics treated in the main body. Its role is to complement the 
localized-damping result and to contrast the two long-time mechanisms.

The paper is organized as follows. In Section~\ref{sec:Striwellpod}, we first establish the equivalence of three Sobolev norms generated by $\Delta$, $\Delta_A$ and $\Delta_G$, respectively. Also, we give the local-in-time Strichartz estimate (with loss of derivatives) for the variable coefficient magnetic Schr\"odinger equation. Using these tools, we then show the local well-posedness of \eqref{equ:DMNLS} in $H^{1+\varepsilon}$. 

In Subsection 2.1, using the uniform bound for the energy $E(u(t))$ established in Subsection \ref{subsec:UnienerLoc}, we prove the global well-posedness which relies on the argument by Burq-G\'erard-Tzvetkov \cite{BGT}. Using the multiplier method by Fanelli-Vega \cite{Fanelli-Vega}, we prove the uniform energy bound in Subsection \ref{subsec:UnienerLoc}. 

In Section \ref{sec:intMor}, we prove the interaction Morawetz estimate involving the damping term and show the scattering. Due to the trapping effect brought by $G$, there is no global-in-time estimate. To overcome this difficulty, we introduce the cut-off function that decompose the solution to $u=\chi u+(1-\chi)u$. For the compactly supported part $\chi u$, we establish the local energy decay to handle it.  In particular, we show that $\chi u\to0$ in $H^s(\R^3)$ with $s<1$. For the remaining part $(1-\chi)u$, it satisfies the cubic nonlinear Schr\"odinger equation with additional external terms involving magnetic and damping potentials $A$ and $ia$, respectively,  and the commutator $[\Delta,\chi]$. Using local smoothing and Strichartz estimates for the free Schr\"odinger group $e^{it\Delta}$ and a perturbation argument, we first prove scattering in $L^2$ and then improve it to any $H^{1-\varepsilon}$ with $\varepsilon>0$ by interpolating with the uniform energy estimate.

Appendix~\ref{appendix_semiclassical} contains the semiclassical functional-calculus details used in Section~\ref{sec:Striwellpod}. Appendix~\ref{appendix:abstract} contains the complementary constant-damping theorem, examples, and proof.

\subsection{Notation}
\label{section_notation}
Define
\begin{align*}%\label{equ:notatFjk}
	%	DA&:=(\partial_jA_k)_{j,k=1}^3,\\
	D_j&:=\pa_j-iA_j,\quad 1\leq j\leq3,\\%\label{equ:notaDelA}
	\Delta_A&:=\nabla_A\cdot\nabla_A=\Delta-2iA\cdot\nabla-i\div A-|A|^2.
\end{align*}
%	Then, it is easy to check that	\begin{align}\label{equ:deride1}		\pa_j(u\bar{v})&=(D_ju)\bar{v}+u\overline{D_jv},\\\label{equ:deride2}		D_jD_k&=iB_{jk}+D_kD_j,\\\label{equ:deride3}		D_j(uv)&=(D_ju)v+u\pa_jv,\,\nabla_A(uv)=(\nabla_Au)v+u\nabla v. 	\end{align}

Let $S^m$ be H\"ormander's symbol class defined by
$$
S^m=\{a\in C^\infty(\R^3\times \R^3)\ |\ |\partial_x^\alpha\partial_\xi^\beta a(x,\xi)|\le C_{\alpha\beta}\langle\xi\rangle^{m-|\beta|}\ \text{for any}\ \alpha,\beta\in \Z_+^3\},
$$
where $m\in \R$ and $\<\xi\>=\sqrt{1+|\xi|^2}$. For $a\in S^m$ and a small parameter $h\in (0,1]$, we associate the semiclassical pseudodifferential operator ($h$-PDO for short): 
$$
a(x,hD)f(x)=\frac{1}{(2\pi h)^3}\int_{\R^6}e^{i(x-y)\cdot\xi/h} a(x,\xi)f(y)dyd\xi.
$$
We refer to monographs \cite{Hormander,DiSj,Stein} for a detailed description on the pseudodifferential operator. 

%%%%%%%%%%%%%%%%%%%%%%%%%%%%%%%%%%%%%%%%%%%%%%%%%%%%%%%%%%%%%%%%%%

%%%%%%%%%%%%%%%%%%%%%%%%%%%%%%%%%%%%%%%%%%%%%%%%%%%%%%%%%%%%%%%%%%

\section{Strichartz estimate and well-posedness}\label{sec:Striwellpod}

In this section, we prove Theorem~\ref{T:gwp}.

To this end, we first recall the equivalence of Sobolev spaces generated by $\Delta$, $\Delta_A$, and $\Delta_G$, and the corresponding Strichartz estimates for the propagator $e^{it\Delta_G}$, and then establish local well-posedness.  %Then, we utilize the Virial argument to show the uniform bound of the energy and local energy decay.  Thus, we obtain the global well-posedness.

\begin{definition}
	For $s\in \R$ and $1<p<\infty$, the spaces $W^{s,p}$, $W^{s,p}_A$, and $W^{s,p}_G$ denote the Sobolev spaces of order $s$ generated by $\Delta$, $\Delta_A$, and $\Delta_G$, respectively. More precisely, they are defined as the completions of $C_c^\infty(\R^3)$ with respect to the following norms
	\begin{align*}
		\|f\|_{W^{s,p}(\R^3)} & =\|(1-\Delta)^\frac s2f\|_{L^p(\R^3)},\\
		\|f\|_{W^{s,p}_A(\R^3)} & =\|(1-\Delta_A)^\frac s2f\|_{L^p(\R^3)},\\
		\|f\|_{W^{s,p}_{G}(\R^3)} &=\|(1-\Delta_{G})^\frac{s}{2}f\|_{L^p(\R^3)}.
	\end{align*}
	We also use the standard notation $H^s=W^{s,2}$, $H_A^s=W_A^{s,2}$, and $H_G^s=W_G^{s,2}$.
	%		Let $s\geq0$, we denote by $H^s$, $H^s$ and $H_{G}^s$ by the completions of 
\end{definition}
Now, we have the equivalence of these Sobolev norms as follows:
\begin{lemma}
	\label{lemma_2_2}
	For any $s\in \R$ and any $1<p<\infty$, the spaces $W^{s,p}$, $W_A^{s,p}$, and $W_G^{s,p}$ coincide with equivalent norms.
	%		For $s\in[0,2]$, we have the following equivalence of Sobolev norms:		\begin{align*}			\|f\|_{{H}^s_{G}}\simeq\|f\|_{{H}_{A}^s}\simeq\|f\|_{H^s}. 		\end{align*}
	%		In particular, we have		\begin{align*}			\|f\|_{{H}_{G}^1}\simeq\| f\|_{{H}_A^1},\,\,\|\langle\Delta_{G}\rangle f\|_{L^2}\simeq\|\langle \Delta_A\rangle f\|_{L^2}.		\end{align*}
\end{lemma}
\begin{proof}
Let
\[
P_0=1-\Delta,\qquad P_A=1-\Delta_A,\qquad P_G=1-\Delta_G.
\]
These are uniformly elliptic second-order operators with smooth bounded coefficients and positive principal symbols. Standard elliptic pseudodifferential calculus for complex powers gives, for every $s\in\mathbb R$,
\[
P_G^{s/2}P_0^{-s/2},\quad P_0^{s/2}P_G^{-s/2},\quad
P_A^{s/2}P_0^{-s/2},\quad P_0^{s/2}P_A^{-s/2}\in\Psi^0.
\]
Their symbols belong to $S^0$, and hence the corresponding operators are bounded on $L^p(\mathbb R^3)$ for every $1<p<\infty$ by the standard order-zero pseudodifferential multiplier theorem. Therefore, for $f\in C_c^\infty(\mathbb R^3)$,
\[
\|P_G^{s/2}f\|_{L^p}\simeq\|P_0^{s/2}f\|_{L^p},
\qquad
\|P_A^{s/2}f\|_{L^p}\simeq\|P_0^{s/2}f\|_{L^p}.
\]
Passing to the completions proves the assertion. We refer to the standard elliptic complex-power calculus in \cite{Hormander,DiSj}.
\end{proof}

%remark
\begin{remark}
	In view of this lemma, we do not distinguish between these three spaces $W^{s,p}$, $W^{s,p}_A$ and $W^{s,p}_G$ (if $p\neq1,\infty$) in the following argument. 
\end{remark}

%We have the following Strichartz estimate. %For a general variable coefficient electric magnetic Schr\"odinger operator $\Delta_G$ with	\begin{gather*}		|\partial_x^\alpha(G(x)-Id)|\leq C_\alpha\langle x\rangle^{-\mu-|\alpha|},\\		|\partial_x^\alpha A_j(x)|\leq C_\alpha\langle x\rangle^{1-\mu-|\alpha|},	\end{gather*}where $\alpha$ is the multi-index and $\mu>0$, Mizutani \cite{Miz13} established the following global-in-space Strichartz estimates with losses of  derivatives.

We next state the Strichartz estimates used in the paper.

\begin{proposition}[Local-in-time Strichartz estimates for the propagator $e^{it\Delta_G}$]\label{prop:Striest}
	Let $s\in \R$ and $(q,r)\in\Lambda_0:=\big\{q,r\geq2\ |\ \tfrac2q=\tfrac32-\tfrac3r\big\}$. Then, for any $T>0$, there holds
	\begin{align}\label{equ:Strihom}
		\|e^{it\Delta_G}f\|_{L_t^q([0,T], W^{s,r}(\R^3))}&\lesssim \<T\>^{1/q} \|f\|_{H^{s+1/q}(\R^3)},\\\label{equ:Striinhom}
		\left\| \int_0^t e^{i(t-\tau)\Delta_G}F(\tau,x)\,d\tau\right\|_{L_t^q([0,T], W^{s,r}(\R^3))}&\lesssim \<T\>^{1/q}\|F\|_{L_t^1([0,T], H^{s+1/q}(\R^3))},
	\end{align}
	where the implicit constants are independent of $T,f$ and $F$. 
\end{proposition}

%proof
\begin{proof}
	%By Lemma \ref{lemma_2_2}, $W^{s,r}$ and $H^{s+1/q}$ can be replaced by $W^{s,r}$ and $H^{s+1/q}$, respectively. 
	Since $e^{it\Delta_G}$ commutes with $(1-\Delta_G)^{s/2}$, taking Lemma \ref{lemma_2_2} into account, it  suffices to prove the case $s=0$. Assume $T>1$ without loss of generality, and  decompose $[0,T]=I_1\cup\cdots\cup I_N$ with $|I_j|=1$ for all $j$ and $N\sim [T]$. It follows from \cite[Theorems 1.2 and 1.5]{Miz13} that
	\begin{align*}
		\|e^{it\Delta_G}f\|_{L_t^q([0,T], L^r)}^q
		&\le \sum_{j=1}^N \|e^{it\Delta_G}f\|_{L_t^q(I_j, L^r)}^q
		\lesssim N	\|f\|_{H^{1/q}}^q
	\end{align*}
	and \eqref{equ:Strihom} follows. By Minkowski's inequality and \eqref{equ:Strihom}, we also obtain
	\begin{align*}
		%\label{equ:Striinhom}
		\left\| \int_0^t e^{i(t-\tau)\Delta_G}F(\tau,\cdot)\,d\tau\right\|_{L_t^q([0,T], L^r)}
		&\lesssim \int_0^T\|e^{i(t-\tau)\Delta_G}F(\tau,\cdot)\|_{L_t^q([0,T], L^r)}d\tau\\
		&\lesssim T^{1/q}\int_0^T\|e^{-i\tau\Delta_G}F(\tau,\cdot)\|_{H^{1/q}}d\tau\\
		&\lesssim T^{1/q}\|F\|_{L_t^1([0,T], H^{1/q})},
	\end{align*}
	where we used the fact that $e^{-i\tau\Delta_G}$ preserves the $H_G^{1/q}$-norm and Lemma \ref{lemma_2_2}. 
	%For general $T>0$ we \eqref{equ:Strihom} and \eqref{equ:Striinhom} follow by first decThe case $s=0$ was proved by  which also implies general cases $s>0$ thanks to Lemma \ref{lemma_2_2} and the fact that $e^{it\Delta_G}$ is  isometric on $W^{s,p}_G$. 
\end{proof}

%remark
\begin{remark}
	The loss of $1/q$ derivatives cannot be removed in general without the non-trapping condition on the metric $G(x)$. This is the reason that we solve \eqref{equ:DMNLS} in $H^{1+\epsilon}$ instead of the energy space $H^1$. 
\end{remark}

%\begin{color}{blue}[Mizutani: Since we don't need Strichartz estimates for $e^{it\Delta_A}$, I simply removed it]\end{color}

As a consequence of Proposition \ref{prop:Striest}, we obtain the local well-posedness of \eqref{equ:DMNLS}.

\begin{proposition}[Local well-posedness]\label{prop:LWP}
	Let $0<\epsilon<\frac12$, $2< q<\frac2{1-2\epsilon}$ and $u_0\in H^{1+\epsilon}(\R^3)$. Then there exist $T>0$ and a unique solution $u\in C([0,T], H^{1+\epsilon}(\R^3))\cap L^q([0,T], L^\infty(\R^3))$ to \eqref{equ:DMNLS} satisfying the following statements: 
	\begin{itemize}
		\item The map $u_0\in H^{1+\epsilon}(\R^3)\mapsto u\in C([0,T], H^{1+\epsilon}(\R^3))$ is continuous. 
		\item If we denote by $T_{\rm max}$ the maximal forward time of existence, then  either  $T_{\rm max}=+\infty$ or $T_{\rm max}<+\infty$ and
		\begin{equation*}%\label{equ:blowcri}
			\lim_{t\to T_{\rm max}}\|u(t)\|_{H^{1+\epsilon}(\R^3)}=+\infty.
		\end{equation*}
		\item If $u_0\in H^r(\R^3)$ for some $r>1+\epsilon$, then $u\in C([0,T],H^r(\R^3))$. 
	\end{itemize}
\end{proposition}

\begin{proof}%We follow the same argument as in \cite[Proposition 3.1]{BGT}. 
	%By Lemma \ref{lemma_2_2}, $H^{1+\epsilon}$ can be replaced by $H^{1+\epsilon}$. 
	Set $s=1+\epsilon$ and $\frac1r=\frac12-\frac{2}{3q}$ so that $\frac2q=\frac32-\frac3r$ and $s-\frac1q>\frac3r$. 
	%\begin{align}\label{prop_LWP_proof_1}W^{s-\frac1q,r}(\R^3)\hookrightarrow L^\infty(\R^3).\end{align}
	Define		$$\Phi(u)=e^{it\Delta_G}u_0-i\int_0^t e^{i(t-s)\Delta_G}\big(iau-|u|^{2}u\big)(s)\,ds$$
	on $X_T=C([0,T], H^s(\R^3))\cap L^q([0,T], W^{s-\frac1q,r}(\R^3))$ equipped with 
	$$
	\|u\|_{X_T}=\|u\|_{L^\infty([0,T],H^s)}+\|u\|_{L^q([0,T],W^{s-1/q,r}_G)}. 
	$$
	We show that $\Phi$ is a contraction on
	$B=\big\{u\in X_T\ |\ \|u\|_{X_T}\leq M\|u_0\|_{H^s(\R^3)}\big\}$, 
	where $M$ is specified later. For $u\in X_T$, Proposition \ref{prop:Striest} implies
	$$
	\|\Phi(u)\|_{X_T}\le (1+C\<T\>^{\frac1q})(\|u_0\|_{H^s}+\|iau-|u|^{2}u\|_{L^1([0,T], H^s)}). 
	$$
	Using the standard nonlinear estimate (see e.g. \cite[Lemma 3.4]{GOV1994})
	\begin{align}
		\label{prop_LWP_proof_2}
		\||u|^2u\|_{H^s}\lesssim \|u\|_{L^\infty}^2\|u\|_{H^s},\quad 0\le s<2, 
	\end{align}
	the embedding $W^{s-\frac1q,r}(\R^3)\hookrightarrow L^\infty(\R^3)$ and H\"older's inequality, we have
	\begin{align*}
		&\big\|iau-|u|^{2}u\big\|_{L^1([0,T], H^s)}\\
		&\lesssim T\|a\|_{W^{2,\infty}}\|u\|_{L^\infty([0,T],H^s)}+T^{1-\frac2q}\|u\|_{L^q([0,T],L^\infty)}^{2}\|u\|_{L^\infty([0,T],H^s)}\\
		&\lesssim T\|u\|_{X_T}+T^{1-\frac2q}\|u\|_{X_T}^3\\
		&\lesssim TM\|u_0\|_{H^s}+T^{1-\frac2q}M^3\|u_0\|_{H^s}^3
	\end{align*}
	and hence, with some $C>0$ independent of $T$ and $u_0$, 
	\begin{align*}
		\|\Phi(u)\|_{X_T}\le (1+C\<T\>^{\frac1q})(1+CTM+CT^{1-\frac2q}M^3\|u_0\|_{H^s}^2)\|u_0\|_{H^s}.
	\end{align*}
	By Proposition \ref{prop:Striest}, we similarly get for $u,v\in X_T$, 
	\begin{align*}
		\|\Phi(u)-\Phi(v)\|_{X_T}%&=\Big\|\int_0^t e^{i(t-s)\Delta_G}\big(ia(u-v)-(|u|^2u-|v|^2v)\big)(s)\,ds\Big\|_{X_T}\\
		&\le (1+C\<T\>^{\frac1q})\|ia(u-v)-(|u|^2u-|v|^2v)\|_{L^1([0,T], H^s(\R^3))}\\
		&\le (1+C\<T\>^{\frac1q})\{CT+C_2T^{1-\frac{2}{q}}(\|u\|_{X_T}^{2}+\|v\|_{X_T}^{2})\}\|u-v\|_{X_T}\\
		%\label{prop_LWP_proof_3}
		&\le  (1+C\<T\>^{\frac1q})(CT+2CT^{1-\frac{2}{q}}M^2\|u_0\|_{H^s}^2)\|u-v\|_{X_T}. 
	\end{align*}
	Let $M=2(1+2C)$. Since $q>2$, one can choose $T=T(\|u_0\|_{H^s})$  small enough such that $\Phi$ is a contraction on $B$. The desired assertion then follows by a standard argument based on the fixed point theorem (see e.g. Cazenave \cite{Caz}). 
\end{proof}

%\begin{color}{blue}[Mizutani: the nonlinear estimates of type \eqref{prop_LWP_proof_2} seem to be used in several places in \cite{BGT} and \cite{LS25} even for $s\ge2$. However, it cannot hold for $s\ge2$ since $$\Delta(|u|^2u)=\Delta(|u|^2)u+2\nabla(|u)^2)\cdot\nabla u+|u|^2\Delta u$$and hence $\|\Delta(|u|^2u)\|_{L^2}$ should involve the term like $\|u\|_{L^\infty}\|\nabla u\|^2_{L^2}$. ]\end{color}

\subsection{Global existence and uniform bounds}

%proposition
\begin{proposition}%[Uniform bound on the mass and energy]
	\label{proposition_uniform}
	Let $T>0$ and $u\in C([0,T],H^{1+\epsilon}(\R^3))$ be a solution to \eqref{equ:DMNLS}. Then there exists $C>0$ independent of $T$ such that 
	\begin{align}
		\label{equ:uniformenergy}
		\sup_{t\in [0,T]}M(u(t))\le M(u_0),\quad \sup_{t\in [0,T]}E(u(t))\le C E(u_0)+CM(u_0),
	\end{align}
	where $M(u)$ and $E(u)$ are defined by \eqref{equ:massG} and \eqref{equ:energyG}. 
\end{proposition}

The proof of this proposition is given in Section \ref{section_local_energy_decay} below.

\begin{proposition}[Global existence]\label{prop:CGT}
	%Let $0<\epsilon<1$. 			Assume that the following holds: there exists a continuous function $\Psi:\R_+\to\R_+^\ast$ \begin{color}{blue}[Mizutani: (1) Using the symbol $C$ would confuse readers. (2) Please write explicitly the definition of $\R_+^*$ instead of using the notation $\R_+^*$]\end{color} such that for any $u_0\in H^{1+\epsilon}(\R^3)$ and $T>0$, if \begin{color}{red}$u\in C([0,T],H^{1+\epsilon}(\R^3))$\end{color} is a solution to \eqref{equ:DMNLS} on $[0,T]$ with initial data $u_0$, then\begin{align}\label{prop:CGT_1}\sup_{[0,T]}\|u(t)\|_{H^1(\R^3)}\leq \Psi\big(\|u_0\|_{H^1(\R^3)}\big).\end{align}		Then, the maximal solutions to \eqref{equ:DMNLS} are global forward in time.
	For any $u_0\in H^{1+\epsilon}(\R^3)$, the maximal solution $u$ constructed in Proposition \ref{prop:LWP} exists globally forward in time, namely $u\in C([0,\infty),H^{1+\epsilon}(\R^3))$. Moreover, $u$ satisfies
	\begin{align}
		\label{equ:uniformenergy_2}
		\sup_{t\ge0}M(u(t))\le M(u_0),\quad \sup_{t\ge0}E(u(t))\lesssim E(u_0)+M(u_0).
	\end{align}
\end{proposition}

%\begin{color}{blue}[Mizutani:Instead of using the abstract condition on $\|u(t)\|_{H^1}$ as in the previous version, I added Proposition \ref{proposition_uniform} and use it in the proof of Proposition \ref{prop:CGT}. With this modification, I also slightly modified the structure of this section. ]\end{color}

To prove this proposition, we first establish the following three lemmas

%lemma
\begin{lemma}
	\label{lemma_CGT_1}
	Let $\varphi\in C_c^\infty(\R)$ and  $1\le p\le q\le \infty$. Then, for any $h\in (0,1]$, 
	\begin{align}
		%\|\varphi(-h^2\Delta_G)f\|_{L^2}&\lesssim h^s \|\varphi(-h^2\Delta_G)f\|_{H^{s}_{G}},\\\label{lemma_CGT_1_1}
		\label{lemma_CGT_1_1}
		\|\varphi(-h^2\Delta_G)f\|_{L^q}&\lesssim h^{-3(1/p-1/q)}\|\varphi(-h^2\Delta_G)f\|_{L^{p}}. 
		%\|\varphi(-h^2\Delta_G)f\|_{L^2}&\lesssim h\|\varphi(-h^2\Delta_G)f\|_{H^1}.
	\end{align}
	Moreover, if in addition $\supp \varphi\Subset (0,\infty)$ and $p>1$, then for any $s\ge0$
	\begin{align}
		%\label{lemma_CGT_1_1}\|\varphi(-h^2\Delta_G)f\|_{L^2}&\lesssim h^s \|\varphi(-h^2\Delta_G)f\|_{H^{s}_{G}},\\\label{lemma_CGT_1_1}
		\label{lemma_CGT_1_2}
		\|\varphi(-h^2\Delta_G)f\|_{L^q}&\lesssim h^{-3(1/p-1/q)+s}\|\varphi(-h^2\Delta_G)f\|_{W^{s,p}}. 
		%\|\varphi(-h^2\Delta_G)f\|_{L^2}&\lesssim h\|\varphi(-h^2\Delta_G)f\|_{H^1}.
	\end{align}
\end{lemma}

%lemma
\begin{lemma}
	\label{lemma_CGT_2}
	Let $\varphi\in C_c^\infty(\R)$ with $\mathop{\mathrm{supp}} \varphi\Subset (0,\infty)$. Then there exists $\alpha>0$ such that, for any $h\in (0,1]$ and interval $I$ with $|I|\le \alpha h$
	\begin{align*}
		%\label{lemma_CGT_2_1}
		\|e^{it\Delta_G}\varphi(-h^2\Delta_G)u_0\|_{L^2(I,L^6)}&\lesssim \|\varphi(-h^2\Delta_G)u_0\|_{L^2},\\
		%\label{lemma_CGT_2_2}
		\left\|\int_{(-\infty,t]\cap I} e^{i(t-s)\Delta_G}\varphi(-h^2\Delta_G)F(s)ds\right\|_{L^2(I,L^6)}
		&\lesssim \|\varphi(-h^2\Delta_G)F\|_{L^1(I,L^2)+L^2(I,L^{\frac 65})}. 
	\end{align*}
\end{lemma}

The proofs of Lemmas \ref{lemma_CGT_1} and \ref{lemma_CGT_2} are essentially same as that of \cite[Corollary 2.2, Propositions 2.9 and Lemma 3.4]{BGT}. We give a brief outline of the proofs in Appendix \ref{appendix_semiclassical} below. 

%lemma
\begin{lemma}
	\label{lemma_CGT_3}
	Let $\varphi\in C_c^\infty(\R)$ with $\mathop{\mathrm{supp}} \varphi\Subset (0,\infty)$ and $b<c$ with $c-b\le1$. Then, for any solution  $v\in C([b,c],L^2)$  to 
	$$
	i\partial_tv+\Delta_Gv=F_1+F_2,\quad F_1\in L^1(I,L^2),\quad F_2\in L^2(I,L^{6/5}),
	$$
	and $h\in (0,1]$, we have
	\begin{align*}
		\|\varphi_hv\|_{L^2([b,c],L^6)}
		&\lesssim \|\varphi_h v\|_{L^\infty([b,c],L^2)}+h^{-1/2}\|\varphi_hv\|_{L^2(I,L^2)}\\
		&\quad+h^{1/2}\|\varphi_hF_1\|_{L^2(I,L^2)}+\|\varphi_hF_2\|_{L^2(I,L^{6/5})},
	\end{align*}
	where we set $\varphi_h=\varphi(-h^2\Delta_G)$ for short. 
\end{lemma}

\begin{proof}
	Recall that $\varphi_hv$ solves
	\begin{align}
		\label{lemma_CGT_3_proof_1}
		\varphi_hv(t)=e^{i(t-t_0)\Delta_G}\varphi_hv(t_0)-i\int_{t_0}^t e^{i(t-s)\Delta_G}\varphi_h(F_1+F_2)(s)ds.
	\end{align}
	for any $t_0\in[b,c]$. Decomposing the time interval $[b,c]$ into $N$-intervals $I_j$ as
	$$
	[b,c]=\bigcup_{j=1}^NI_j,\ N\sim h^{-1},\ 0<|I_j|\le \frac{\alpha h}{2},\ I_j\cap I_k=\emptyset\ \text{if}\ j\neq k,\ b\in I_1,\ c\in I_N, 
	$$
	we apply Lemma \ref{lemma_CGT_2} with $I=I_1$ to \eqref{lemma_CGT_3_proof_1} with $t_0=b$, obtaining
	\begin{align}
		\label{lemma_CGT_3_proof_2}
		\|\varphi_h v\|_{L^2(I_1,L^6)}
		&\lesssim \|\varphi_hv(b)\|_{L^2}+\|\varphi_h F_1\|_{L^1(I_1,L^2)}+\|\varphi_hF_2\|_{L^2(I_1,L^{6/5})}\\
		\nonumber
		&\lesssim \|\varphi_hv(b)\|_{L^2}+h^{1/2}\|\varphi_h F_1\|_{L^2(I_1,L^2)}+\|\varphi_hF_2\|_{L^2(I_1,L^{6/5})}.
	\end{align}
	%where we also used the bound $\|av\|_{H^1}\lesssim \|a\|_{W^{2,\infty}}\|v\|_{H^1}$. 
	For $j=N$, we similarly obtain by using \eqref{lemma_CGT_3_proof_1} with $t_0=c$
	\begin{align}
		\|\varphi_h v\|_{L^2(I_N,L^6)}
		%\nonumber&\lesssim \|\varphi_hv(a)\|_{L^2}+\|f\|_{L^2(I_1,L^{6/5})}\\
		\label{lemma_CGT_3_proof_3}
		&\lesssim \|\varphi_hv(c)\|_{L^2}+h^{1/2}\|\varphi_h F_1\|_{L^2(I_N,L^2)}+\|\varphi_hF_2\|_{L^2(I_N,L^{6/5})}.
	\end{align}
	For $2\le j\le N-1$, we let $J_j$ an $\frac{\alpha h}{4}$-neighborhood of $I_j$ such that $J_j\cap J_k=\emptyset$ for $|j-k|\ge2$ and take $\chi_j\in C_c^\infty(\R)$ satisfying $\mathop{\mathrm{supp}} \chi_j\subset J_j$, $\chi_j\equiv1$ on $I_j$ and $\|\chi_j'\|_{L^\infty}\lesssim h^{-1}$. Then $\chi_j(t)\varphi_hv$ satisfies
	$$
	(i\partial_t+\Delta_G)\chi_j(t)\varphi_hv=i\chi_j'\varphi_hv+\chi_j\varphi_h(F_1+F_2),\quad \chi_j(a)\varphi_hv(a)=0,
	$$
	and  thus the Duhamel formula
	$$
	\chi_j(t)\varphi_hv(t)=-i\int_{a}^t e^{i(t-s)\Delta_G}\left\{i\chi_j'(s)\varphi_hv(s)+\chi_j(s)\varphi_h(F_1+F_2)(s)\right\}ds.
	$$
	This, together with Lemma \ref{lemma_CGT_2}, implies
	\begin{align}
		\label{lemma_CGT_3_proof_4}
		&\|\varphi_hv\|_{L^2(I_j,L^6)}
		\le \|\chi_j\varphi_hv\|_{L^2(J_j,L^6)}\\
		\nonumber
		&\lesssim \|\chi'\|_{L^\infty}\|\varphi_hv\|_{L^1(J_j,L^2)}+\|\varphi_hF_1\|_{L^1(I_j,L^2)}+\|\varphi_hF_2\|_{L^2(J_j,L^{6/5})}\\
		\nonumber
		&\lesssim h^{-1/2}\|\varphi_h v\|_{L^2(J_j,L^2)}+h^{1/2}\|\varphi_hF_1\|_{L^1(I_j,L^2)}+\|\varphi_hF\|_{L^2(J_j,L^{6/5})}. 
	\end{align}
	It follows from \eqref{lemma_CGT_3_proof_2}--\eqref{lemma_CGT_3_proof_4} that
	\begin{align*}
		\|\varphi_h v\|_{L^2(I,L^6)}
		&\le \Big(\sum_{j=1}^N\|\varphi_h v\|_{L^2(I_j,L^6)}^2\Big)^{1/2}\\
		&\lesssim \|\varphi_h v(b)\|_{L^2}+\|\varphi_h v(c)\|_{L^2}+h^{-1/2}\|\varphi_hv\|_{L^2(I,L^2)}\\
		&\quad+h^{1/2}\|\varphi_hF_1\|_{L^2(I,L^2)}+\|\varphi_hF_2\|_{L^2(I,L^{6/5})},
	\end{align*}
	which implies the desired assertion. 
\end{proof}

\begin{proof}[Proof of Proposition \ref{prop:CGT}]
	Let $s=1+\epsilon$. We may assume $0<\epsilon<1$. Thanks to the blow-up alternative in Proposition \ref{prop:LWP} and the uniform estimates \eqref{equ:uniformenergy}, the global existence follows from the claim that
	$
	\sup_{t\in [0,T]}\|u(t)\|_{H^s}
	$ remains bounded as long as $T$ remains bounded. Once we have the global existence, \eqref{equ:uniformenergy_2} follows from \eqref{equ:uniformenergy}. The proof of this claim follows the same line as that of \cite[Theorem 3]{BGT}. We first prove this claim by assuming $u\in C([0,T],L^\infty(\R^3))$ and then, for general cases, by an approximation argument which goes back to Yudovitch \cite{Yudovich}. 
	
	At first, for any time interval $I\subset [0,T]$ with $|I|\le1$, we prove
	\begin{align}
		\label{prop:CGT_proof_1}
		\|u\|_{L^2(I,L^\infty)}\lesssim 1+|I|^{1/2}\{\log(2+\|u\|_{L^2(I, H^s)})\}^{1/2}
	\end{align}
	for any $H^{s}$ solution, where the implicit constant depends on $u_0$ through the norm $\|u_0\|_{H^1}$ only and independent of $I$.
	To this end, applying Lemma \ref{lemma_CGT_3} to $u$ with $(F_1,F_2)=(-iau,|u|^2u)$ and using Lemma \ref{lemma_CGT_1} with $s=1$, we have for any bounded interval $I\subset [0,T]$ (without size restriction)
	\begin{align*}%\label{refine-strichartz}
		\|\varphi_hu\|_{L_t^2(I,L^6)}
		&\lesssim \|\varphi_hu\|_{L^\infty(I,L^2)}+h^{-1/2}\|\varphi_hu\|_{L^2(I,L^2)}\\
		&\quad+ h^{1/2}\|\varphi_h au\|_{L^2(I,L^2)}+\|\varphi_h|u|^2u\|_{L^2(I,L^{6/5})}\\
		&\lesssim h\|\varphi_hu\|_{L^\infty(I,H^1)}+h^{1/2}\|\varphi_hu\|_{L^2(I,H^1)}\\
		&\quad+ h^{3/2}\|\varphi_h au\|_{L^\infty(I,H^1)}+h\|\varphi_h|u|^2u\|_{L^\infty(I,W^{1,6/5})}, 
	\end{align*}
	where the implicit constants depend on $I$ through its length $|I|$ only. Since the first, third and last terms of the RHS satisfies
	$$
	\|\varphi_hu\|_{H^1}+\|\varphi_h au\|_{H^1}\lesssim (1+\|a\|_{W^{1,\infty}})\|u\|_{H^1},
	$$
	and
	\begin{align*}
		\|\varphi_h|u|^2u\|_{W^{1,6/5}}
		&\lesssim \||u|^2u\|_{L^{6/5}}+\||u|^2\nabla u\|_{L^{6/5}}\\
		&\lesssim \||u|^2\|_{L^3}\|u\|_{L^2}+\||u|^2\|_{L^3}\|\nabla u\|_{L^2}\\
		&\lesssim \|u\|_{L^6}^2\|u\|_{H^1}\lesssim \|u\|_{H^1}^3,
	\end{align*}
	if we write $\Lambda(u)=\|u\|_{L^\infty(I,H^1)}+\|u\|_{L^\infty(I,H^1)}^3$, then
	\begin{align}
		\label{refine-strichartz}
		\|\varphi_hu\|_{L_t^2(I,L^6)}\lesssim h^{1/2}\|\varphi_hu\|_{L^2(I,H^1)}+h\Lambda(u),\quad h\in (0,1].
	\end{align}
	Here we should stress that if $|I|\le1$ then the implicit constant is independent of $I$. With this estimate, Lemma \ref{lemma_CGT_1} and the uniform bound
	$$
	\sup_{t\in [0,T]}\|u(t)\|_{H^1}\lesssim E(u_0)+M(u_0)
	$$
	which follows from Proposition \ref{proposition_uniform}, one can follow exactly the same argument as that in the proof of \cite[Thoerem 3]{BGT} to obtain \eqref{prop:CGT_proof_1}. 
	
	Now we use the additional condition $u\in C([0,T],L^\infty)$ to observe that the Duhamel formula for $u$ and \eqref{prop_LWP_proof_2} imply
	$$
	\|u\|_{L^\infty(I,H^s)}\lesssim \|u(b)\|_{H^s}+\int_I(1+\|u(t)\|_{L^\infty}^2)\|u(t)\|_{H^s}ds,\quad  b:=\min I.
	$$
	Then, Gronwall's lemma and \eqref{prop:CGT_proof_1} yield
	\begin{align*}
		\|u\|_{L^\infty(I,H^s)}
		\le C\|u(b)\|_{H^s}\exp(C\|u\|_{L^2(I,L^\infty)}^2)
		%&\le C\|u(b)\|_{H^s}\exp(C\{1+|I|\log(2+\|u\|_{L^2(I,H^s)})\})\\
		\le C\|u(b)\|_{H^s}(2+\|u\|_{L^2(I,H^s)})^{C|I|}
	\end{align*}
	with some $C>0$ independent of $I$. If $C|I|\le1/2$, then 
	$$\gamma^2-C^2\|u(b)\|_{H^s}^2\gamma -2C^2\|u(b)\|_{H^s}^2\le0,$$
	where $\gamma=\|u\|_{L^\infty(I,H^s)}^2$. Solving this quadratic inequality shows
	\begin{align*}
		\|u\|_{L^\infty(I,H^s)}\le \tilde C (\|u(b)\|_{H^s}+\|u(b)\|_{H^s}^2).
	\end{align*}
	with some $\tilde C>0$. Applying this estimate successively starting from the interval $I=[0,\frac{1}{2C}]$, we conclude that $\sup_{t\in [0,T]}\|u(t)\|_{H^s}
	$ remains bounded as $T$ remains bounded. This completes the proof for smooth solutions. 
	
	%For non-smooth solutions $u$ (which do not necessarily belong to $C([0,T],L^\infty)$), a standard approximation argument and the persistence of regularity in Proposition \ref{prop:LWP} imply that there exists a family $\{u^\varepsilon\}$ of smooth solutions so that, by applying the above argument to $u^\varepsilon$, 

	%By regularization argument, $u^\varepsilon$ is smooth and global. In particular, one can prove that $u^\varepsilon\to u$ and the limit equation belongs to $L_{\rm loc}^\infty H^s\cap L^2_{\rm loc}L^\infty$. \begin{color}{blue}[Mizutani: the above two sentences do not make sense]\end{color}

	To deal with non-smooth solutions, we first show the uniqueness of weak $H^1$-solutions via Yudovich's argument \cite{BGT,Yudovich}, which consists of the following two steps: 
	\begin{enumerate}
		\item[(i)] For any given $M,T>0$ and arbitrary $\delta>0$, there exists $\varepsilon>0$ such that if $u$ and $\tilde{u}$ are two (mild) solutions to \eqref{equ:DMNLS} on $[0,T]$ with initial data $u(0)$ and $\tilde{u}(0)$ respectively, such that $$\|u(0)-\tilde{u}(0)\|_{L^2(\R^3)}\leq\varepsilon$$ and 
		\begin{align}
			\label{Yudovich_1}
			\|u\|_{L^\infty((0,T),H^1(\R^3))}\leq M,\quad 
			\|\tilde u\|_{L^\infty((0,T),H^1(\R^3))}\leq M,
		\end{align}
		then we have
		\begin{align*}
			\sup_{t\in[0,T]}\big\|u(t)-\tilde{u}(t)\big\|_{L^2(\R^3)}\leq\delta.
		\end{align*}    
		\item[(ii)] The solution $u$ to \eqref{equ:DMNLS} is indeed unique in $L^\infty ((0,T],H^1(\R^3))$. 
	\end{enumerate}
	We will prove Claim (i) first. To show this, we first prove that there exists $t_0>0$ depending only on $M$ such that, for any $\delta>0$, there exists $\varepsilon>0$ such that
	\begin{align}\label{reduction1}
		\|\tilde{u}(0)-u(0)\|_{L^2}\leq\varepsilon\Longrightarrow	\sup_{t\in[0,t_0]}\|\tilde{u}(t)-u(t)\|_{L^2}<\delta.
	\end{align}With \eqref{reduction1} at hand, since $t_0$ is independent of $T$, we can split the interval $[0,T]$ into $t_0^{-1}T$ pieces of subintervals. Then Claim (i) follows directly. To deduce \eqref{reduction1}, we need an $L^2L^{2p}$ estimate of the form 
	\begin{align}\label{refine-strichartz-2}
		\|u\|_{L_t^2([0,T],L^{2p}(\R^3))}\leq C(\sqrt{Tp}+1)
	\end{align}
	for any $p\ge6$, which can be derived by using \eqref{refine-strichartz}, \eqref{prop:CGT_proof_1} and summing over $k$ where  $h=2^{-k}$ (see \cite[Proof of Theorem 3]{BGT}). Now we start to prove \eqref{reduction1}. Define
	\begin{align*}
		g(t):=\|u(t)-\tilde{u}(t)\|_{L^2}^2.
	\end{align*}
	Since $a\ge0$, we have for $p>2$, 
	\begin{align*}
		\frac{d}{dt}g(t)&\le 2\Im \langle|u(t)|^2u(t)-|\tilde{u}(t)|^2\tilde{u}(t),u(t)-\tilde{u}(t)\rangle_{L^2}\\
		&\lesssim \int_{\R^3}|u(t)-\tilde{u}(t)|^2(|u(t)|^2+|\tilde{u}(t)|^2)\\
		&\lesssim \| u(t)-\tilde{u}(t)\|_{L^{\frac{2p}{p-1}}(\R^3)}^2\big(\| u(t)\|_{L^{2p}(\R^3)}^2+\|\tilde{u}(t)\|_{L^{2p}(\R^3)}^2\big). 
	\end{align*}
	Since $2<\frac{2p}{p-1}<6$, by interpolating between $L^2$ and $L^6$, where the latter norm is bounded by the embedding $H^1\subset L^6$ and the hypothesis \eqref{Yudovich_1}, we get
	\begin{align*}
		\frac{d}{dt}g(t)\leq C(\|u(t)\|_{L^{2p}}^2+\|\tilde u(t)\|_{L^{2p}}^2)g(t)^{1-\frac{3}{2p}}.
	\end{align*}
	with some $C>0$ depending on $M$. Integrating in $t$, we have
	\begin{align}\label{integral}
		pg(t)^\frac{3}{2p}\lesssim(\|u(t)\|_{L_t^2L^{2p}}^2+\|\tilde u(t)\|_{L_t^2L^{2p}}^2)+pg(0)^\frac{3}{2p}.
	\end{align}
	On the other hand, using \eqref{refine-strichartz-2}, we obtain
	\begin{align*}
		pg(t)^\frac{3}{2p}\leq C(tp+1)+pg(0)^\frac{3}{2p}.
	\end{align*}
	Then, we have
	\begin{align*}
		g(t)\leq \Big(C(t+\frac1p)+g(0)\Big)^\frac{2p}{3}\leq 2^\frac{2p}{3}g(0)+\big(2C(t+\frac1p)\big)^\frac{2p}{3}.
	\end{align*}
	For the sufficiently small $t_0$ and $t\in[0,t_0]$, we get
	\begin{align*}
		g(t)\leq2^\frac{2p}{3}g(0)+\big(\frac{1}{2}+\frac{2C}{p}\big)^\frac{2p}{3}.
	\end{align*}
	Taking $p$ large enough and, then taking $\varepsilon=\varepsilon(p)$ small enough, we have
	\begin{align*}
		g(t)\leq \frac{\delta}{2}+2^\frac{2p}{3}\varepsilon^2\le \delta. 
	\end{align*}
	This completes the proof of Claim (i). Taking $g(0)=0$, a similar argument as above and density, we can obtain that $u=\tilde{u}$ in $L^\infty H^1$. 
	
	Finally, we utilize a regularization argument to treat the case $u_0\in H^s$ for $s>1$. Let $u\in C([0,T_{\mathrm{max}}),H^s)$ be a maximal solution to \eqref{equ:DMNLS}. We can approximate $u_0$ in $H^s$ by a sequence  $\{u_0^n\}_{n\ge1}\subset H^k$ with some large $k>3/2$. Let $u^n$ be the associated solution to \eqref{equ:DMNLS}. It follows from the previous argument for smooth solutions that $\{u^n\}$ is bounded in $L^\infty([0,\infty), H^1) \cap L_{\mathrm{loc}}^\infty([0,\infty), H^s)$. Using Claim (i), we obtain that $\{u^n\}$ is Cauchy in $C([0,T],L^2)$ for any $T>0$. Since $u^{n}_{0} \to u_0$ in $L^{2}$, it follows that $u^n\to v$ in $C([0,T], L^{2})$ with some $v\in L^\infty([0,\infty), H^1) \cap L_{\mathrm{loc}}^\infty([0,\infty), H^s)$. By using the interpolation argument twice, we observe that, for any $T>0$, $u^n\to v$ in $C([0,T], H^\sigma)$ for any $0\le \sigma<1$ and then for any $0\le \sigma<s$. Finally, it follows from the Sobolev embedding that $|u^n|^2u^n-|v|^2v\to 0$ in $C([0,\infty),L^2)$. Therefore, $v$ is a weak $H^1$-solution to \eqref{equ:DMNLS}. By the above uniqueness, we conclude $u=v$.
\end{proof}

%subsection		
\subsection{Uniform energy bound and local energy decay}\label{subsec:UnienerLoc}
\label{section_local_energy_decay}	
This section is devoted to the proof of Proposition \ref{proposition_uniform}. We also prove a local energy decay used in the proof of Theorem~\ref{T:main}. 

We begin with the following basic mass and energy identities: 
\begin{lemma}[mass and energy identities]\label{lem:maenerg}
	Let $u\in C([0,T], H^{1+\epsilon})$ be a solution to \eqref{equ:DMNLS}. Then, for any $0\leq t_0\leq t<T$, 
	\begin{align}\label{equ:mass}
		M(u(t))-M(u(t_0))&=-2\int_{t_0}^t\int_{\R^3}a(x)|u(s,x)|^2\,dx\,ds,\\
		\label{equ:energy}
		E(u(t))-E(u(t_0))&
		=-\int_{t_0}^t\int_{\R^3}a(x)\big(G\nabla_A u\cdot \overline{\nabla_Au}+|u|^{4}\big)\,dx\,ds\\\nonumber
		&-\mathop{\mathrm{Re}} \int_{t_0}^t\int_{\R^3}\bar{u} G\nabla a\cdot\nabla_Au\,dx\,ds\\\nonumber
		&=-\int_{t_0}^t\int_{\R^3}a(x)\big(G\nabla_A u\cdot \overline{\nabla_Au}+|u|^{4}\big)\,dx\,ds\\\nonumber
		&+\frac12\mathop{\mathrm{Re}} \int_{t_0}^t\int_{\R^3}|u|^2\nabla\cdot G\nabla a\,dx\,ds.
	\end{align}
	In particular,
	\begin{equation}\label{equ:axmasscon}
		\int_0^{T}\int_{\R^3}a(x)|u(s,x)|^2\,dx\,ds\leq\frac{M(u_0)}2.
	\end{equation}
	
\end{lemma}

\begin{proof}
	For smooth solutions, \eqref{equ:mass} is obtained by multiplying \eqref{equ:DMNLS} by $\overline u$, integrating in $t,x$ and taking the imaginary part, while \eqref{equ:energy} can be verified by multiplying \eqref{equ:DMNLS} by $\overline{\partial_t u}$, integrating in $t,x$ and taking the real part. Moreover, these identities also hold for general $C([0,T], H^{1+\epsilon})$ solutions by a standard approximation argument. Since $M(u(t))\le M(u_0)$ by \eqref{equ:mass}, \eqref{equ:axmasscon} follows. 
\end{proof}	

Consider next the function $u$ solves
\begin{equation}\label{equ:nls}
	i\pa_tu+\Delta_A u=|u|^2u+F(t,x),\quad (t,x)\in\R\times\R^3.
\end{equation}
For a radially symmetric real-valued function $\rho$ satisfying $\nabla \rho\in L^\infty(\R^3)$, we define the Morawetz action
\begin{equation}\label{equ:moract}
	M_\rho(t):=2\Im \int_{\R^3}(\nabla \rho)(x)\cdot (\nabla_A
	u)(x)\bar{u}(x)\,dx.
\end{equation}
By the same argument as in \cite[Theorem 1.2]{Fanelli-Vega} and \cite[Lemma 3.1]{CCL}, we get
\begin{lemma}[Virial identity for \eqref{equ:nls}]\label{lem:moriden}
	There holds
	\begin{align*}
		\frac{d}{dt}M_\rho(t)&=\int_{\R^3}(-\Delta^2\rho)|u|^2dx
		+4\int_{\R^3}(\nabla^2\rho) \nabla_Au\cdot\overline{\nabla_Au}dx
		+\int_{\R^3}(\Delta\rho)|u|^4\,dx\\
		&+4\Im \int_{\R^3}\rho_ruB_\tau \cdot \overline{\nabla_A u}dx+2\sum_{j=1}^3\int_{\R^3}\rho_j \{F,u\}_P^jdx
	\end{align*}
	where $\rho_j=\partial_{x_j}\rho$, $\nabla^2\rho=(\partial_j\partial_k\rho)_{j,k=1}^3$, $\rho_r=\frac{x}{|x|}\cdot\nabla\rho$ and
	$$\{f,g\}_P=\Re(f\overline{\nabla_A g}-g\overline{\nabla_Af}).$$ % % The red term is new, we should be careful to control this term. Maybe one needs to use Luca-Vega's local smoothing estimate to control this term.
\end{lemma}

%remark
\begin{remark}
	We have used the Coulomb gauge condition $\div A=0$ to derive this Virial identity; otherwise, several  additional terms which dot have sign appear. 
\end{remark}

We shall apply this lemma to \eqref{equ:DMNLS}. Since
\begin{align*}
	\{iau,u\}_P&=-2a\Im (u\overline{\nabla_Au})
\end{align*}
and, by the symmetry of $G$, 
\begin{align}\label{equ:errtermMore}	
	&2\int_{\R^3}\nabla\rho\cdot\{(\Delta_A-\Delta_G)u,u\}_P\,dx\\\nonumber
	&=2\mathop{\mathrm{Re}} \int_{\R^3}(\nabla\Delta\rho)\cdot(G-I)\bar{u}\nabla_Au\,dx+4\mathop{\mathrm{Re}} \int_{\R^3}(G-I)\nabla_Au\cdot(\nabla^2\rho)\overline{\nabla_Au}\\\nonumber
	&-2\mathop{\mathrm{Re}} \int_{\R^3}\rho_i\overline{D_ju}\pa_iG_{jk}D_ku\,dx\\\nonumber
	&=: \epsilon(t), 
\end{align}
we choose $F=-iau+(\Delta_A-\Delta_G)u$ to obtain

\begin{lemma}[Virial identity for \eqref{equ:DMNLS}]\label{lem:Moraw}
	Let $u_0\in H^{1+\epsilon}(\R^3)$ and $u\in C([0,T), H^{1+\epsilon})$ be a solution to \eqref{equ:DMNLS}. Then, for any $0\leq s\leq t<T$, there holds
	\begin{align*}\nonumber
		\frac{d}{dt}M_\rho(t)&=\int_{\R^3}\big(-\Delta^2\rho\big)|u|^2dx+4\int_{\R^3}(\nabla^2\rho) \nabla_A u\cdot \overline{\nabla_A u}dx+\int_{\R^3}(\Delta \rho)|u|^{4}dx\\
		&+
		4\Im \int_{\R^3}a\rho_ru\frac{x}{|x|}\cdot\overline{\nabla_A u}dx
		%		\label{equ:perMor}
		+4\Im \int_{\R^3}\rho_ruB_\tau \cdot \overline{\nabla_A u}dx
		+ \epsilon(t),
	\end{align*}
	where $\epsilon(t)$ is defined as above and satisfies
	$$| \epsilon(t)|\lesssim \int_{ \mathop{\mathrm{supp}} (G-I)}\big(|\nabla u(t,x)|^2+|u(t,x)|^2\big)\,dx.$$
	
\end{lemma}
%
%{\bf Analysis:} If we choose $\rho(x)=\sqrt{1+|x|^2}$, then we have
%\begin{equation*}
%-\Delta^2\rho(x)=\frac{(d-1)(d-3)}{(1+|x|^2)^\frac32}+\frac{6(d-3)}{(1+|x|^2)^\frac52}+\frac{15}{(1+|x|^2)^\frac72}.
%\end{equation*}
%When $d=3$, we get
%$$-\Delta^2\rho(x)=\frac{15}{(1+|x|^2)^\frac72},$$
%which is accepted.

We now prove Proposition \ref{proposition_uniform}.

%proof
\begin{proof}[Proof of Proposition \ref{proposition_uniform}]
	The uniform bound $M(u(t))\le M(u_0)$ easily follows from \eqref{equ:mass}  since $a(x)\ge0$. To deal with the energy $E(u(t))$, we combine the argument by \cite{LS25} with the multiplier technique by \cite{Fanelli-Vega}. For a constant $R>0$, we define
	\begin{align*}
		\rho(x):=R\rho_0\left(\frac rR\right),
	\end{align*}
	where $r=|x|$ and, with some constant $c_0>0$ specified later, $\rho_0$ is given by
	\begin{align*}
		\rho_0(r)=\int_0^x\rho_0'(s)\,ds,\quad \rho_0'(r):=\begin{cases}
			c_0+\frac13r,&r\leq1,\\
			c_0+\frac12-\frac{1}{6r^2},&r>1. 
		\end{cases}
	\end{align*}
	The following basic properties of $\rho$ will be used frequently: 
	\begin{align*}
		\rho_r
		&=\begin{cases}
			c_0+\frac{r}{3R}&\text{for}\ r\leq R,\\
			c_0+\frac12-\frac{R^2}{6r^2}&\text{for}\ r>R,
		\end{cases}\quad 
		\rho_{rr}
		=\begin{cases}
			\frac{1}{3R}&\text{for}\ r\leq R,\\
			\frac{R^2}{3r^3}&\text{for}\ r>R,
		\end{cases}\\
		\Delta\rho&=\rho_{rr}+2r^{-1}\rho_r\ge0,\quad
		\Delta^2\rho=-4\pi\delta_{x=0}-R^{-2}\delta_{|x|=R}.
	\end{align*}
	We also recall the following formula:
	\begin{align*}
		(\nabla^2\rho) \nabla_A u\cdot \overline{\nabla_A u}
		=\rho_{rr}|\nabla_A^ru|^2+r^{-1}\rho_r|\nabla_A^\tau u|^2,
	\end{align*}
	where $\nabla_A^ru=|x|^{-2}(x\cdot\nabla_A u)x$ and $\nabla_A^\tau u$ is the tangential part of $\nabla_A^\tau u$ so that 
	$$|\nabla_Au|^2=|\nabla_A^ru|^2+|\nabla_A^\tau u|^2,\quad \nabla_A^ru\cdot \nabla_A^\tau u=0,\quad B_\tau\cdot\nabla_Au=B_\tau\cdot\nabla_A^\tau u.$$
	In particular, 
	$$
	(\nabla^2\rho) \nabla_A u\cdot \overline{\nabla_A u}\ge \begin{cases}\frac{1}{3R}|\nabla_Au|^2+\frac{c_0}{|x|}|\nabla_A^\tau u|^2,& |x|\le R,\\\frac{c_0}{|x|}|\nabla_A^\tau u|^2,& |x|\ge R.
	\end{cases}
	$$
	
	Since $a\ge0$ and $\mathop{\mathrm{supp}}a$ is compact, we observe by using \eqref{equ:energy} that
	\begin{align}\label{energy_estimate}
		E(u(t))-E(u_0)
		&\le \frac12\int_0^t\int_{\R^3}\big|\nabla\cdot(G\nabla a)\big||u|^2\,dx\,ds\\
		\nonumber
		&\lesssim \int_0^t\int_{|x|\le R_0}|u|^2\,dx\,ds\\
		\nonumber
		&\lesssim \int_0^t\int_0^{R_0} \int_{|x|=r}|u|^2\,d\omega\,dr\,ds\\
		\nonumber
		&\lesssim \sup_{R>0}\frac{1}{R^2}\int_0^t\int_{|x|=R}|u|^2\,d\omega\,ds,
	\end{align}
	where $\int_{|x|=r}|u|^2d\omega$ denotes the surface integral over the sphere $\{|x|=r\}$ and we have fixed a constant $R_0>1$ such that
	\begin{align}
		\label{supp_a}
		\supp a\subset \{|x|\le R_0-1\}.
	\end{align}
	Let
	$$
	\lambda_1(t):=\sup_{R>0}\frac{1}{R^2}\int_0^t\int_{|x|=R}|u|^2\,d\omega\,ds. 
	$$
	To obtain the desired bound for $E(u(t))$, it is enough to show 
	\begin{equation}\label{equ:lamunibound}
		\sup_{t\in [0,T]}\lambda_1(t)\lesssim E(u_0)+M(u_0).
	\end{equation}
	To this end, using \eqref{equ:moract}, the bound $|\nabla\rho|\le c_0+1/2$ and \eqref{energy_estimate}, we have
	\begin{equation}\label{equ:uppbounMrho}
		\sup_{R>0}\big|M_\rho(t)\big|\leq CM(u(t))^\frac12 E(u(t))^\frac12\leq C_1M(u_0)^\frac12\big(E(u_0)+C_0\lambda_1(t)\big)^\frac12. 
	\end{equation}
	Next, we claim that $M_\rho(t)$ has the following lower bound
	\begin{equation}\label{equ:Mlowbound}
		\sup_{R>0}M_\rho(t)\geq C_2\lambda_1(t)-C_3\big(M(u_0)+E(u_0)\big).
	\end{equation}
	With this claim and the upper bound \eqref{equ:uppbounMrho} at hand, we get
	\begin{align*}
		C_2\lambda_1(t)-C_3\big(M(u_0)+E(u_0)\big)\leq C_1M(u_0)^\frac12\big(E(u_0)+C_0\lambda_1(t)\big)^\frac12.
	\end{align*}
	Hence, with some constants $b,c>0$, 
	$$\lambda_1(t)^2-b\big(M(u_0)+E(u_0)\big)\lambda(t)-c \big(M(u_0)+E(u_0)\big)^2\leq 0,$$
	which implies \eqref{equ:lamunibound}. It remains to prove \eqref{equ:Mlowbound}. Lemma \ref{lem:Moraw}, the above properties of  $\rho$ and the support property $\supp a\subset \{|x|\le R_0\}$ imply
	\begin{align}\label{proof_energy_1}
		M_\rho(t)
		%\nonumber
		& \ge \frac{1}{R^2}\int_0^t\int_{|x|=R}|u|^2\,d\omega\,ds
		+\frac{4}{3R}\int_0^t\int_{|x|\le R}|{\nabla_A u}|^2dx ds\\
		\nonumber
		&\quad +4c_0\int_0^t\int_{\R^3}\frac{|\nabla_A^\tau u|^2}{|x|}dx ds 
		-4\int_0^t\int_{|x|\le R_0}|a\rho_ru\frac{x}{|x|}\cdot\overline{\nabla_A u}|dx ds\\
		\nonumber
		&\quad -4\int_0^t\int_{\R^3}|\rho_ruB_\tau \cdot \overline{\nabla_A u}|dx ds
		-\int_0^t |\epsilon(s)|ds+M_\rho(0),
	\end{align}
	where the first three terms of the RHS are positive and 
	\begin{align}
		\label{proof_energy_2}
		\sup_{R>0}M_\rho(0)\lesssim M(u_0)^{1/2}E(u_0)^{1/2}\lesssim M(u_0)+E(u_0)
	\end{align}
	We set for short
	\begin{align*}
		%D_1&=\sup_{R>0}\frac{1}{R^2}\int_0^t\int_{|x|=R}|u|^2\,d\omega\,ds,\\
		\lambda_2(t)=\sup_{R>0}\frac{4}{3R}\int_0^t\int_{|x|\le R}|{\nabla_A u}|^2dx ds,\quad\lambda_3(t)=\int_0^t\int_{\R^3}\frac{|\nabla_A^\tau u|^2}{|x|}dx ds.
	\end{align*}
	
	For the fourth term of the RHS in \eqref{proof_energy_1}, we use H\"older's inequality, the bound $\rho_r\le c_0+1/2$, \eqref{equ:axmasscon} and Young's inequality to obtain, for any $\delta>0$, 
	\begin{align}
		\label{proof_energy_3}
		&4\int_0^t\int_{|x|\le R_0}|a\rho_r\bar{u}\frac{x}{|x|}\cdot\nabla_A u|dx ds
		\le 4(c_0+1/2)\int_0^t\int_{|x|\le R_0}a|u||\nabla_Au|\,dx\,ds\\\nonumber
		& \le4(c_0+1/2)\|a\|_{L^\infty}^{1/2} \left(\int_0^t\int_{\R^3}a|u|^2\,dx\,ds\right)^{1/2}\left(\int_0^t\int_{|x|\le R_0}|\nabla_Au|^2\,dx\,ds\right)^{1/2}\\\nonumber
		&\le C M(u_0)+\frac{4\delta}{3R_0}\int_0^t\int_{|x|\le R_0}|\nabla_Au|^2\,dx\,ds\\\nonumber
		&\le C M(u_0)+ \delta \lambda_2(t)
	\end{align}
	with some $C=C(c_0,\delta)>0$ independent of $R$. 
	
	For the fifth term related with $B_\tau$, we first deal with the case when the smallness condition \eqref{equ:Btaucond} holds. The case with the control condition \eqref{control_B} will be discussed in the end of the proof. Taking the fact $B_\tau\cdot \nabla_Au=B_\tau\cdot\nabla^\tau_Au$ into account, we have
	\begin{align}
		\label{proof_energy_4}
		&4 \int_0^t\int_{\R^3}|\rho_ruB_\tau \cdot \overline{\nabla_A u}|dx ds\\\nonumber
		&\le (4c_0+2)\left(\int_0^t\int_{\R^3}|x||B_\tau|^2|u|^2dx\right)^{1/2}\left(\int_0^t\int_{\R^3}\frac{|\nabla^\tau_Au|^2}{|x|}dx\right)^{1/2}\\\nonumber
		&\le (4c_0+2)\left(\int_0^t\int_0^\infty \int_{|x|=r}|x||B_\tau|^2|u|^2d\omega\,dr\,ds\right)^{1/2}\lambda_3(t)^{1/2}\\\nonumber
		&\le (4c_0+2) \left(\int_0^\infty \sup_{|x|=r}|x|^3|B_\tau|^2 dr\right)^{1/2}\left(\sup_{R>0}\frac{1}{R^2}\int_0^t \int_{|x|=R}|u|^2d\omega\right)^{1/2}\lambda_3(t)^{1/2}\\\nonumber
		&=(4c_0+2)\||x|^{\frac32}B_\tau\|_{L^2_rL^\infty(S_r)} \lambda_1(t)^{1/2}\lambda_3(t)^{1/2}.
	\end{align}
	
	For the term involving $\epsilon(t)$, we follow the same argument as in \cite{LS25} for which the control condition \eqref{equ:assumpona} plays a crucial role.  The condition \eqref{equ:assumpona}, the energy identity \eqref{equ:energy} and \eqref{equ:axmasscon} imply
	\begin{align*}
		\int_0^t |\epsilon(s)|ds
		&\lesssim \int_0^t\int_{\supp(G-I)}\left(|\nabla_Au|^2+|u|^2\right)\,dx\,ds\\
		&\lesssim \int_0^t\int_{\R^3}a\left(G\nabla_Au\cdot\nabla_Au+|u|^2+|u|^4\right)\,dx\,ds\\
		&\lesssim \int_0^t\int_{\R^3}\left(|\overline uG\nabla a\cdot\nabla_Au|+a|u|^2\right)\,dx\,ds+E(u_0)-E(u(t))\\
		&\lesssim \int_0^t\int_{\R^3}|\overline uG\nabla a\cdot\nabla_Au|\,dx\,ds+E(u_0)+M(u_0)
	\end{align*}
	By  \cite[Lemma 4.4]{YNC}, we know that for any $\delta>0$, there exists $C_\delta>0$ such that
	\begin{equation}\label{equ:nablabound}
		|\nabla a|\leq C_\delta a+\delta\chi_{\supp a+\{|x|\le1\}}.
	\end{equation}
	%where $\psi\in C_c^\infty(\R^3)$ is supported in $\supp a+\{|x|\le1\}\subset \{|x|\le R_0\}$, $\psi\equiv 1$ on $\supp a$ and $0\leq\psi\leq1$. 
	Since $\supp a+\{|x|\le1\}\subset \{|x|\le R_0\}$, the same argument as as above shows
	\begin{align}
		\label{proof_energy_5}
		&\int_0^t\int_{\R^3}|\overline uG\nabla a\cdot\nabla_Au|\,dx\,ds\\\nonumber
		&\le  \|G\|_{L^\infty}\int_0^t\int_{|x|\le R_0}(C_\delta a+\delta)|u||\nabla_Au|\,dx\,ds\\\nonumber
		&\lesssim C_\delta\int_0^t\int_{|x|\le R_0}a|u|^2\,dx\,ds+\delta \int_0^t\int_{|x|\le R_0}|\nabla_Au|^2\,dx\,ds
		+\delta\int_0^t\int_{|x|\le R_0}|u|^2\,dx\,ds\\\nonumber
		&\le C_\delta M(u_0)+\delta \lambda_1(t)+\delta\lambda_2(t).
	\end{align}
	
	It follows from \eqref{proof_energy_1}--\eqref{proof_energy_5} that for any $\delta>0$ 
	\begin{align*}
		\sup_{R>0}M_\rho(t)&\ge \big(\tfrac12-\delta\big)\lambda_1+c_1(\delta)\lambda_2+4c_0\lambda_3\\
		&\quad -(4c_0+2)\||x|^{\frac32}B_\tau\|_{L^2_rL^\infty(S_r)} \lambda_1^{1/2}\lambda_3^{1/2}-C\{M(u_0)+E(u_0)\}
	\end{align*}
	with $c_1(\delta)>0$ and $C=C(\delta,c_0)>0$. In the above estimate, we use the following basic fact 
    $$\sup_{R>0}(f_R+g_R)\geq\frac12\big(\sup_{R>0}f_R+\sup_{R>0}g_R\big),\,\,f_R,g_R\geq0.$$
    Now we observe that, under the smallness condition \eqref{equ:Btaucond}, one can choose $0<\delta<1/4$ uniformly in $\lambda_1,\lambda_2$ in such a way that
	\begin{align}
		\label{proof_energy_6}
		(\tfrac12-\delta)\lambda_1+4c_0\lambda_3-(4c_0+2)\||x|^{\frac32}B_\tau\|_{L^2_rL^\infty(S_r)} \lambda_1^{1/2}\lambda_3^{1/2}\ge \delta(\lambda_1+\lambda_3).
	\end{align}
	Therefore, 
	\begin{align}
		\label{proof_energy_7}
		\sup_{R>0}M_\rho(t)\ge \delta\lambda_1+(1-2\delta)\lambda_2+\delta\lambda_3-C\{M(u_0)+E(u_0)\}
	\end{align}
	and \eqref{equ:Mlowbound} follows under the smallness \eqref{equ:Btaucond}.

	Finally, we observe that \eqref{equ:Btaucond} was used only for the estimate \eqref{proof_energy_6}. If we assume the control condition \eqref{control_B} instead of \eqref{equ:Btaucond}, then
	$$
	\int_0^t\int_{\R^3}|\rho_ruB_\tau \cdot \overline{\nabla_A u}|dx ds\lesssim \int_0^t\int_{\supp B_\tau}(|\nabla_Au|^2+|u|^2)\,dx\,ds.
	$$
	Hence this term can be dealt with the same way as that for $\epsilon(t)$, yielding 
	$$
	\int_0^t\int_{\R^3}|\rho_ruB_\tau \cdot \overline{\nabla_A u}|dx ds\le C_\delta M(u_0)+\delta\lambda_1(t)+\delta\lambda_2(t). 
	$$
	This, together with the above argument, shows
	$$
	\sup_{R>0}M_\rho(t)\ge \delta\lambda_1+(1-2\delta)\lambda_2+4c_0\lambda_3-C\{M(u_0)+E(u_0)\}
	$$
	and \eqref{equ:Mlowbound} follows. This completes the proof. 
\end{proof}

Thanks to Propositions \ref{proposition_uniform} and \ref{prop:CGT}, we have obtained the global solution $u\in C([0,\infty),H^{1+\epsilon}(\R^3))$ to \eqref{equ:DMNLS} satisfying the global energy bound \eqref{equ:uniformenergy_2}. As an immediate corollary, we derive the following: 

\begin{proposition}[Local energy decay]\label{prop:uniformenergy} Let $u\in C([0,\infty),H^{1+\epsilon}(\R^3))$ be the global solution obtained in Proposition \ref{prop:CGT}. Then
	\begin{align}
		\label{equ:locmasdec}
		&\sup_{R>0}\frac{1}{R^2}\int_0^\infty\int_{|x|=R}|u|^2d\omega\,dt+\sup_{R>0}\frac{1}{R}\int_0^\infty\int_{|x|\le R}|{\nabla_A u}|^2\,dx\,dt\\\nonumber&\quad +\int_0^\infty\int_{\R^3}\frac{|\nabla_A^\tau u|^2}{|x|}\,dx\,dt\lesssim   E(u_0)+M(u_0). 
	\end{align}
	Moreover, for any $\chi\in C_0^\infty(\R^3)$ and $0\le s<1$, 
	\begin{equation}\label{equ:locenedecHs}
		\lim_{t\to\infty}\|\chi u(t)\|_{H^s(\R^3)}=0.
	\end{equation}
\end{proposition}

\begin{proof}
	% \eqref{equ:locmasdec} follows by letting $t\to \infty$ in \eqref{proof_energy_7}. %By \eqref{equ:locmasdec}, there exists $t_n$ with $t_n\nearrow \infty$ as $n\to \infty$ such that $\chi u(t_n)\to 0$ in $L^2$ as $\to n\to\infty$. This sequential convergence, together with $M(t)\le M(t_n)$ for any $t\ge t_n$ by \eqref{equ:mass}, shows 
%     The proof of \eqref{equ:locenedecHs} is the same as that of \cite[Corollary 3.6]{LS25}. 
% \end{proof}
% \begin{proof}
Estimate \eqref{equ:locmasdec} follows by letting $t\to\infty$ in \eqref{proof_energy_7}.
For \eqref{equ:locenedecHs},
let $\chi\in C_c^\infty(\mathbb R^3)$ be real-valued and choose
$R_\chi>0$ such that
\[
\operatorname{supp}\chi\subset B(0,R_\chi).
\]
By the coarea formula and the first term in \eqref{equ:locmasdec},
\[
\begin{aligned}
\int_0^\infty\int_{B(0,R_\chi)}|u|^2\,dx\,dt
&=
\int_0^{R_\chi}
\int_0^\infty\int_{|x|=r}|u|^2\,d\omega\,dt\,dr
\\
&\le
\frac{R_\chi^3}{3}
\sup_{R>0}\frac1{R^2}
\int_0^\infty\int_{|x|=R}|u|^2\,d\omega\,dt
<\infty.
\end{aligned}
\]
The second term in \eqref{equ:locmasdec}, evaluated at $R=R_\chi$, also gives
\[
\int_0^\infty\int_{B(0,R_\chi)}
|\nabla_Au|^2\,dx\,dt<\infty.
\]
Hence
\[
\chi u\in L^2\bigl((0,\infty);H^1(\mathbb R^3)\bigr),
\]
where we have used the boundedness of $A$ and Lemma \ref{lemma_2_2}

Set
\[
f(t):=\|\chi u(t)\|_{L^2}^2.
\]
Using equation \eqref{equ:DMNLS}, integration by parts gives the localized
mass identity
\[
\frac12 f'(t)
=
-\int_{\mathbb R^3}a\chi^2|u|^2\,dx
+
\operatorname{Im}\int_{\mathbb R^3}
\overline u\,G\nabla_Au\cdot\nabla(\chi^2)\,dx.
\]
Consequently, by the boundedness of $G$ and $a$,
\[
|f'(t)|
\lesssim_\chi
\int_{B(0,R_\chi)}
\bigl(|u(t,x)|^2+|\nabla_Au(t,x)|^2\bigr)\,dx.
\]
The preceding local spacetime estimates therefore imply
\[
f,f'\in L^1(0,\infty).
\]
Thus $f(t)$ has a finite limit as $t\to\infty$. Since
$f\ge0$ and $f\in L^1(0,\infty)$, this limit must be zero:
\[
\|\chi u(t)\|_{L^2}\longrightarrow0.
\]

Finally, the global bound \eqref{equ:uniformenergy_2}, the ellipticity condition \eqref{equ:Gpd},
and Lemma \ref{lemma_2_2} imply
\[
\sup_{t\ge0}\|\chi u(t)\|_{H^1}<\infty.
\]
Therefore, for every $0\le s<1$, interpolation yields
\[
\|\chi u(t)\|_{H^s}
\lesssim
\|\chi u(t)\|_{L^2}^{1-s}
\|\chi u(t)\|_{H^1}^{s}
\longrightarrow0
\qquad (t\to\infty).
\]
This proves \eqref{equ:locenedecHs}.
\end{proof}

\section{Interaction Morawetz estimate and proof of scattering}\label{sec:intMor}

In this section, we establish the interaction Morawetz estimate under the full-curvature hypotheses of Assumption~\ref{assum:scattering}, and then prove Theorem~\ref{T:main}.

\subsection{Interaction Morawetz estimates}

\begin{proposition}[Interaction Morawetz estimate]\label{prop:intMorest}
	Assume the hypotheses of Theorem~\ref{T:gwp} and Assumption~\ref{assum:scattering}. Let $u\in C([0,\infty);H^{1+\epsilon}(\R^3))$ be the corresponding global solution. Then
	\begin{equation}\label{equ:intMora}
		\|u\|_{L^4([0,\infty)\times\R^3)}^4
		\le C\big(M(u_0),E(u_0),A,\|\mathcal{F}\|_{L^\infty}\big),
	\end{equation}
where $\mathcal{F}=dA$ denote the curvature of the magnetic field.
\end{proposition}

\begin{proof}	We first recall the magnetic interaction Morawetz identity established in \cite{CCL}.  Given a radially symmetric function $\rho$, we define the Morawetz action centered at $y$ and the associated interaction Morawetz action by
	\begin{align}\label{equ:moractre1}
		M_\rho^y(t)&:=2\Im \int_{\R^3}\nabla \rho(x-y)\cdot \nabla_A
		u(t,x)\bar{u}(t,x)\,dx,\\
		\label{equ:moractre12}
		M(t)&:=\int_{\R^3}|u(t,y)|^2 M_\rho^y(t)\,dy\\\nonumber
		&=2{\rm
			Im}\iint_{\R^3\times\R^3}|u(t,y)|^2\nabla \rho(x-y)\cdot \nabla_A
		u(t,x)\bar{u}(t,x)\,dx\,dy.
	\end{align}
	Set
	\[
	\mathcal R:=-iau+(\Delta_A-\Delta_G)u,
	\qquad
	p_k^A(t,x):=\Im\big(\overline{u(t,x)}D_ku(t,x)\big).
	\]
	Applying the one-particle virial identity with the translated multiplier $x\mapsto\rho(x-y)$ and differentiating the $y$-density gives
	\begin{align}
		\frac{d}{dt}M(t)
		&=-4\iint(\nabla^2\rho)(x-y)\Im(\bar u\nabla_Au)(x)\cdot\Im(\bar u\nabla_Au)(y)\,dx\,dy\label{equ:main1}\\
		&\quad+4\iint\Im(\mathcal R\bar u)(y)\nabla\rho(x-y)\cdot\Im(\bar u\nabla_Au)(x)\,dx\,dy\label{equ:main2}\\
		&\quad+\iint(-\Delta^2\rho)(x-y)|u(x)|^2|u(y)|^2\,dx\,dy\label{equ:main3}\\
		&\quad+4\iint|u(y)|^2(\nabla^2\rho)(x-y)\nabla_Au(x)\cdot\overline{\nabla_Au(x)}\,dx\,dy\label{equ:main4}\\
		&\quad+\iint|u(y)|^2(\Delta\rho)(x-y)|u(x)|^4\,dx\,dy\label{equ:main5}\\
		&\quad+2\iint|u(y)|^2\rho_j(x-y)\{\mathcal R,u\}_P^j(x)\,dx\,dy\label{equ:errtermMor}\\
		&\quad-4\sum_{j,k=1}^3\iint|u(y)|^2\rho_j(x-y)\Fcurv_{kj}(x)p_k^A(x)\,dx\,dy.\label{equ:main5-new}
	\end{align}
	The last line is the full magnetic-curvature contribution. The minus sign corresponds to the convention $D_j=\partial_j-iA_j$; its sign will be irrelevant because the term is estimated in absolute value.

	We now choose $\rho(z)=|z|$. Hence
	\[
	\nabla_x\rho(x-y)=\frac{x-y}{|x-y|},\qquad
	\Delta_x\rho(x-y)=\frac2{|x-y|},
	\]
	\[
	\nabla_x^2\rho(x-y)
	=\left(\frac{\delta_{jk}}{|x-y|}
	-\frac{(x_j-y_j)(x_k-y_k)}{|x-y|^3}\right)_{j,k},
	\qquad
	\Delta_x^2\rho(x-y)=-8\pi\delta_{x=y}.
	\]
	It follows from \cite[Section~4]{CCL} that
	\begin{equation}\label{equ:lowbtvz}
		\eqref{equ:main1}+\eqref{equ:main3}+\eqref{equ:main4}+\eqref{equ:main5}
		\ge3\|u(t)\|_{L_x^4}^4.
	\end{equation}
	Consequently,
	\begin{align}\label{equ:Maimt}
		\frac{d}{dt}M(t)
		\ge3\|u(t)\|_{L_x^4}^4+\mathrm I(t)+\mathrm{II}(t)+\mathrm{III}_{\Fcurv}(t)(t),
	\end{align}
	where
	\begin{align*}
	\mathrm I(t)
	&:=4\iint\Im(\mathcal R\bar u)(y)\nabla\rho(x-y)\cdot\Im(\bar u\nabla_Au)(x)\,dx\,dy,\\
	\mathrm{II}(t)
	&:=2\iint|u(y)|^2\rho_j(x-y)\{\mathcal R,u\}_P^j(x)\,dx\,dy,\\
	\mathrm{III}_{\Fcurv}(t)
	&:=-4\sum_{j,k=1}^3\iint|u(y)|^2
	\frac{x_j-y_j}{|x-y|}\Fcurv_{kj}(x)p_k^A(x)\,dx\,dy.
	\end{align*}
	{\bf Estimate of the term $\textrm{I}$:}  We write
	\begin{align*}
		\textrm{I}&=-4\iint a(y)|u(y)|^2\nabla \rho(x-y)\cdot \Im \big(\bar{u}\nabla_A u \big)(x)\,dx\,dy\\
		&\quad +4\iint \Im \big[ (\Delta_A-\Delta_G)u\bar{u}\big](y)\nabla \rho(x-y)\cdot \Im \big(\bar{u}\nabla_Au\big)(x)\,dx\,dy\\
		&=: \textrm{I}_1+\textrm{I}_2.
	\end{align*}
	By $|\nabla\rho|\le1$, Cauchy--Schwarz, and \eqref{equ:uniformenergy_2}, uniformly in $y$ we have
	\[
	|M_\rho^y(t)|\le2\|u(t)\|_{L^2}\|\nabla_Au(t)\|_{L^2}
	\le C\big(M(u_0),E(u_0)\big).
	\]
	Consequently, $|M(t)|\le M(u_0)\sup_y|M_\rho^y(t)|\le C(M(u_0),E(u_0))$.
	This bound, together with \eqref{equ:axmasscon}, implies
	\begin{align*}%\label{equ:I10t}
		\int_0^{\infty} |\textrm{I}_1(t)|\,dt
		\le 4\int_0^\infty \int_{\R^3} a(y)|u(y)|^2\big|M_\rho^y(t)\big|\,dy
		\le  C\big(M(u_0),E(u_0)\big). %\int_0^{+\infty}\int  a(y)|u(y)|^2\,dy\,dt\le  C\big(M(u_0),E(u_0)\big).
	\end{align*}
	Since Hardy's inequality yields
	$$
	\||x-y|^{-1}u\|_{L^2_x}\lesssim \||x|^{-1}u(\cdot -y)\|_{L^2_x}\lesssim \|u\|_{H^1}
	$$
	uniformly in $y\in \R^3$, by integrating by parts in the $y$-variable and applying H\"older's inequality, we obtain
	\begin{align*}
		|\textrm{I}_2|&=4\Big| \iint \Im \big[\bar{u}(G-I)\nabla_A u\big](y)\nabla^2\rho(x-y)\cdot \Im \big(\bar{u}\nabla_Au\big)(x)\,dx\,dy\Big|\\
		&\le C\iint_{y\in\mathop{\mathrm{supp}} (G-I)}\big|(\nabla^2)\rho(x-y)\big|\cdot|u(x)|\cdot|\nabla_A u(x)|\cdot|u(y)|\cdot|\nabla_A u(y)|\,dx\,dy\\
		&\lesssim \int_{y\in\mathop{\mathrm{supp}} (G-I)}\big\||x-y|^{-1}u\big\|_{L^2_x}
		\|\nabla_A u\|_{L^2}|u(y)|\cdot|\nabla u(y)|\,dy\\
		&\lesssim \|u\|_{H^1}^2\int_{y\in\mathop{\mathrm{supp}} (G-I)}\big(|\nabla_A u(y)|^2+|u(y)|^2\big)\,dy.
	\end{align*}
	This together with Proposition \ref{prop:uniformenergy} implies
	\begin{align*}
		\int_0^{\infty} |\textrm{I}_2(t)|\,dt&\lesssim  \int_0^{+\infty} \int_{y\in\mathop{\mathrm{supp}} (G-I)}\big(|\nabla_A u(y)|^2+|u(y)|^2\big)\,dy\,dt\\
		&\le C\big(M(u_0),E(u_0)\big).
	\end{align*}
	Therefore
	\begin{align*}%\label{equ:Itest1}
		\int_0^{\infty} |\textrm{I}(t)|\,dt \le C\big(M(u_0),E(u_0)\big).
	\end{align*}
	{\bf Estimate of the term $\textrm{II}$:}
	We write
	\begin{align*}
		\textrm{II}&=-2\iint |u(y)|^2\rho_j(x-y)\big\{iau,u\big\}_P^j(x)\,dx\,dy\\
		&+2\iint  |u(y)|^2\rho_j(x-y)\big\{(\Delta_A-\Delta_G)u,u\big\}_P^j(x)\,dx\,dy\\
		&=: \textrm{II}_1+\textrm{II}_2.
	\end{align*}
	Noting that $\big\{iau,u\big\}_P=-2a\Im (u\overline{\nabla_Au})$, we obtain
	\begin{align*}
		\big|\textrm{II}_1\big|
		&=4\Big|\iint |u(y)|^2\rho_j(x-y)a(x)\Im (u\overline{\nabla_Au})(x)\,dx\,dy\Big|\\
		&\lesssim M(u_0)\int_{\mathop{\mathrm{supp}}  a}\big(|\nabla_A u|^2+|u|^2\big)\,dx.
	\end{align*}
	%
	%Noting that
	%$$\big\{(\Delta-\Delta_G)v,v\big\}_P=-\nabla \mathop{\mathrm{Re}} \big\{(\Delta-\Delta_G)v\bar{v}\big\}+2\mathop{\mathrm{Re}} \big\{ (\Delta-\Delta_G)v\nabla\bar{v}\big\},$$
	Using \eqref{equ:errtermMore}, we derive
	\begin{align*}
		\textrm{II}_2
		&=-2\mathop{\mathrm{Re}} \iint |u(y)|^2(\nabla\Delta \rho)(x-y)\cdot \left((G-I)\bar{u}\nabla_A u\right)(x)\,dx\,dy\\
		&-4\mathop{\mathrm{Re}} \iint |u(y)|^2(\nabla^2\rho)(x-y) ((G-I)\nabla_A u)(x)\cdot \overline{\nabla_Au}(x)\,dx\,dy\\
		&+2\sum_{k=1}^3\iint|u(y)|^2 \rho_k(x-y)(\pa_kG)(x)\nabla_Au(x)\cdot \overline{\nabla_Au}(x)\,dx\,dy.
	\end{align*}
	This, together with the same argument as above using Hardy's inequality, implies
	\begin{align*}
		|\textrm{II}_2|
		&\lesssim \iint_{x\in \supp (G-I)} |x-y|^{-2}|u(y)|^2|u(x)||\nabla_Au(x)|\,dx\,dy\\
		&\quad+\iint_{x\in \supp (G-I)}|x-y|^{-1}|u(y)|^2|\nabla_Au(x)|^2\,dx\,dy\\
		&\quad +\iint_{x\in \supp(G-I)}|u(y)|^2|\nabla_Au(x)|^2\,dx\,dy\\
		&\lesssim \|u\|_{H^1}^2\int_{\mathop{\mathrm{supp}} (G-I)}\big(|\nabla_A u|^2+|u|^2\big)\,dx.
	\end{align*}
	Hence Proposition \ref{prop:uniformenergy} yields
	\begin{align*}
		\int_0^{\infty}|\textrm{II}|\,dt&\lesssim C\big(M(u_0),E(u_0)\big).
	\end{align*}
	\textbf{Estimate of the full-curvature term $\mathrm{III}_{\Fcurv}$.}
	Since
	\[
	|p^A(t,x)|\le |u(t,x)|\,|\nabla_Au(t,x)|
	\]
	and $|(x-y)/|x-y||=1$, the mass bound gives
	\begin{equation}\label{equ:III-basic}
	|\mathrm{III}_{\Fcurv}(t)|
	\lesssim M(u_0)\int_{\R^3}|dA(x)|\,|u(t,x)|\,|\nabla_Au(t,x)|\,dx.
	\end{equation}
By Remark \ref{magnetic-condition}, the magnetic curvature $\Fcurv$ satisfies the following bounds
\begin{align*}
				\mathcal K_1(\Fcurv)&:=\sum_{j\in\Z}2^{j+1}\|\Fcurv\|_{L^\infty(C_j)}^{2-2\beta}<\infty,\\
				\mathcal K_2(\Fcurv)&:=\int_0^\infty r^2\|\Fcurv\|_{L^\infty(S_r)}^{2\beta}\,dr<\infty;
			\end{align*}
Arguing like the proof in \cite{CCL},  using
	\[
	|dA|\,|u|\,|\nabla_Au|
	\le\frac12|dA|^{2-2\beta}|\nabla_Au|^2
	+\frac12|dA|^{2\beta}|u|^2,
	\]
	the dyadic decomposition and the definitions of $\lambda_1(T)$ and $\lambda_2(T)$ yield
	\begin{align}
	\int_0^T\!|\mathrm{III}_{\Fcurv}(t)|\,dt
	&\lesssim M(u_0)\left[
	\mathcal K_1(\Fcurv)\lambda_2(T)
	+\mathcal K_2(\Fcurv)\lambda_1(T)\right]\label{equ:III-K}
	\end{align}
where $\mathcal{F}=dA$.
Accordingly, we split the full-curvature contribution as
\[
\int_0^T |III_{\Fcurv}(t)|\,dt \lesssim III_{{\Fcurv},a}+III_{{\Fcurv},b},
\]
where
\[
III_{{\Fcurv},a} \lesssim M(u_0)K_1({\Fcurv})\lambda_2(T),
\qquad
III_{{\Fcurv},b} \lesssim M(u_0)K_2({\Fcurv})\lambda_1(T).
\]
	Indeed,
\begin{align*}
\operatorname{III}_{{\Fcurv},a} &= \|u_0\|_{L^2}^2 \int_0^T \sum_{j \in \mathbb{Z}} \int_{C_j} |dA(x)|^{2-2\beta} |\nabla_A u(x)|^2 \, dx dt \\
&\le \|u_0\|_{L^2}^2 \sum_{j \in \mathbb{Z}} \sup_{x \in C_j} 2^{j+1} |dA(x)|^{2-2\beta} \int_0^T \int_{C_j} \frac{|\nabla_A u(x)|^2}{2^{j+1}} dx dt \\
&\le \|u_0\|_{L^2}^2 \sum_{j \in \mathbb{Z}} \sup_{x \in C_j} 2^{j+1} |dA(x)|^{2-2\beta} \lambda_2(T),
\end{align*}
	while the coarea formula gives the estimate of second term
\begin{align*}
    \operatorname{III}_{{\Fcurv},b} &= \|u_0\|_{L^2}^2 \int_0^T \int_0^\infty \int_{|x|=R} R^2 |dA(x)|^{2\beta} \frac{|u(x)|^2}{R^2} d\sigma dR dt \\
&\le \|u_0\|_{L^2}^2 \left( \int_0^\infty \sup_{|x|=R} |x|^2 |dA(x)|^{2\beta} dR \right) \lambda_1(T).
\end{align*}

Combining with Proposition~\ref{prop:uniformenergy}, it holds
	\begin{equation}\label{equ:III-L1}
	\int_0^\infty(|\mathrm{III}_{{\Fcurv},a}(t)|+|\operatorname{III}_{{\Fcurv},b}(t)|)\,dt
	\le C\big(M(u_0),E(u_0),\|\Fcurv\|_{L^\infty}\big)<\infty,
	\end{equation}
    where $C$ depends on the  compactly supported set of $\Fcurv$.

	Plugging the above estimates for $\mathrm I$, $\mathrm{II}$, and $\mathrm{III}_{dA}$ into \eqref{equ:Maimt}, we conclude
	\begin{align*}
		\big\|u\big\|_{L_{t,x}^4(\R^+\times\R^3)}^4\leq \sup_{t>0}\big| M(t)\big|+ C\big(M(u_0),E(u_0)\big)\leq C_1\big(M(u_0),E(u_0),\|\Fcurv\|_{L^\infty}\big).
	\end{align*}
\end{proof}

\subsection{Proof of scattering}
Here we prove Theorem~\ref{T:main}.

First, it is enough to establish scattering in $L^2(\R^3)$. Indeed, if the $L^2$-scattering holds, namely, the limit
$$
u_+=\lim_{t\to \infty}e^{-it\Delta}u(t)
$$
in $L^2(\R^3)$ exists, then Theorem~\ref{T:gwp} implies
\[
\sup_{t>0}\|u(t)\|_{H^1}^2\lesssim M(u_0)+E(u_0).
\]
Then the uniqueness of the weak limit in $H^1(\R^3)$ shows $u_+\in H^1(\R^3)$. By the interpolation, we therefore conclude $\|u(t)-e^{it\Delta}u_+\|_{H^s}\to 0$ as $t\to \infty$ for all $0\le s<1$. 

To prove the $L^2$-scattering, we let $\chi\in C_c^\infty(\R^3)$ satisfy $\chi\equiv1$ on a neighbourhood of $\operatorname{supp}a$. By virtue of \eqref{equ:locmasdec}, it suffices to show there exists $u_+\in L^2(\R^3)$ such that
\begin{equation}\label{equ:1-chismall}
	\lim_{t\to\infty}\|(1-\chi)u(t)-e^{it\Delta}u_+\|_{L^2}=0.
\end{equation}
Set $w=(1-\chi)u$. Since $\operatorname{supp}(G-I)\Subset\{a>0\}$ and $\chi\equiv1$ on a neighbourhood of $\operatorname{supp}a$, we have $G=I$ on $\operatorname{supp}(1-\chi)$. Moreover, $\div A=0$, and hence
\[
\Delta_A=\Delta-2iA\cdot\nabla_A+|A|^2.
\]
Using
\[
\Delta((1-\chi)u)=(1-\chi)\Delta u-2\nabla\chi\cdot\nabla u-(\Delta\chi)u
\]
and $\nabla u=\nabla_Au+iAu$, we obtain
\begin{equation}\label{equ:w-exterior}
 i\partial_tw=-\Delta w+V_1\cdot\nabla_Au+V_2u+(1-\chi)|u|^2u,
\end{equation}
where
\begin{equation}\label{equ:V1V2-correct}
 V_1=-2\nabla\chi+2i(1-\chi)A,
 \qquad
 V_2=-\Delta\chi-2i\nabla\chi\cdot A-(1-\chi)|A|^2-i(1-\chi)a.
\end{equation}
Since $\chi=1$ on $\operatorname{supp}a$, the last term in $V_2$ vanishes identically. Assumption~\ref{assum:scattering}, together with the compact support of $\nabla\chi$ and $\Delta\chi$, implies
\begin{align}
 \sum_{j\ge0}2^{3j}\|V_1\|_{L^\infty(C_j)}^2&<\infty,
 &
 \sum_{j\ge0}2^{5j}\|V_2\|_{L^\infty(C_j)}^2&<\infty,
 \label{proof_scattering_1}
\end{align}
where $C_0=\{|x|\le1\}$ and $C_j=\{2^j\le|x|\le2^{j+1}\}$ for $j\ge1$. For the $|A|^2$ contribution, note that if $a_j=2^{3j}\|A\|_{L^\infty(C_j)}^2$, then
\[
\sum_{j\ge1}2^{5j}\|A\|_{L^\infty(C_j)}^4
=\sum_{j\ge1}2^{-j}a_j^2<\infty.
\]
By the Duhamel principle, we obtain
\begin{align*}
	w(t)&=e^{it\Delta}\Big[(1-\chi)u_0-i\int_0^t e^{-is\Delta}\big(V_1\cdot \nabla_Au+V_2u+(1-\chi)|u|^{2}u\big)(s)\,ds\Big]\\
	&=: e^{it\Delta}u_+(t).
\end{align*}
We show that $u_+(t)$ is Cauchy in $L^2(\R^3)$ as $t\to \infty$. Indeed, we have
\begin{align*}
	\big\|u_+(t_1)-u_+(t_2)\big\|_{L^2}
	&\le \Big\|\int_{t_1}^{t_2} e^{-is\Delta}(V_1\cdot \nabla_A+V_2)u\,ds\Big\|_{L^2}%+\Big\|\int_{t_1}^{t_2} e^{-is\Delta}V_2u\,ds\Big\|_{L^2}\\
	+\Big\|\int_{t_1}^{t_2} e^{-is\Delta}|u|^2u\,ds\Big\|_{L^2}\\
	%		&\quad +\Big\|\int_{t_1}^{t_2} e^{-is\Delta_A}[\chi,\Delta_A]u(s)\,ds\Big\|_{H^\frac12}\\
	&=: I_1+I_2.
\end{align*}
To estimate $I_1$, we use the weighted $L^2_{t,x}$ estimate of the form
\begin{align}
	\label{proof_scattering_2}
	\Big\|\int_0^\infty e^{-is\Delta}f(s)\,ds\Big\|_{L^2_x}\lesssim \||x|f\|_{L^2_tL^2_x},
	%\label{proof_scattering_2}\Big\||\nabla|^{\frac12}\int_0^te^{-is\Delta}F(s)\,ds\Big\|_{L^2_x}&\lesssim \sum_{j\in \Z}^\infty2^{j/2}\|F_j\|_{L^2_tL^2_x},
\end{align}
which is the dual of the standard weighted $L^2_{t,x}$-estimate (see \cite[Theorem 1]{KatoYajima}):
$$
\||x|^{-1}e^{it\Delta}u_0\|_{L^2_tL^2_x}\lesssim \|u_0\|_{L^2_x}. 
$$
By \eqref{proof_scattering_1} and \eqref{proof_scattering_2},
\begin{align}
 I_1^2
 &\lesssim
 \int_{t_1}^{t_2}\!\int_{\R^3}|x|^2
 \big(|V_1|^2|\nabla_Au|^2+|V_2|^2|u|^2\big)\,dx\,dt.
 \label{proof_scattering_3}
\end{align}
The integrand on the right-hand side is integrable on $[0,\infty)\times\R^3$. Indeed, Proposition~\ref{prop:uniformenergy} gives
\[
\int_0^\infty\!\int_{C_j}|\nabla_Au|^2\,dx\,dt
\lesssim2^j\lambda_2(\infty)
\]
and, by the coarea formula,
\[
\int_0^\infty\!\int_{C_j}|u|^2\,dx\,dt
\lesssim2^{3j}\lambda_1(\infty).
\]
Consequently,
\begin{align*}
&\int_0^\infty\!\int |x|^2|V_1|^2|\nabla_Au|^2
\lesssim\lambda_2(\infty)\sum_{j\ge0}2^{3j}\|V_1\|_{L^\infty(C_j)}^2<\infty,\\
&\int_0^\infty\!\int |x|^2|V_2|^2|u|^2
\lesssim\lambda_1(\infty)\sum_{j\ge0}2^{5j}\|V_2\|_{L^\infty(C_j)}^2<\infty.
\end{align*}
The absolute continuity of the integral in \eqref{proof_scattering_3} therefore yields
\[
I_1\longrightarrow0
\qquad\text{as }t_1,t_2\to\infty.
\]

To estimate $I_2$, we use the endpoint Strichartz estimate (see \cite{KeelTao})
\begin{align*}
	%\label{proof_scattering_3}
	\Big\|\int_0^\infty e^{-is\Delta}f(s)\,ds\Big\|_{L^2_x}\lesssim \|f\|_{L^2_tL^{\frac65}_x},
\end{align*}
the Sobolev embedding $W^{1,1}(\R^3) \subset L^{\frac65}(\R^3)$, H\"older's inequality and \eqref{equ:uniformenergy_2} to obtain
\begin{align*}
	I_2&\lesssim \||u|^{2}u\|_{L^2([t_1,t_2],L^{\frac65})}\\
	&\lesssim \||u|^{2}u\|_{L^2([t_1,t_2],W^{1,1})}\\
	&\lesssim \|u\|_{L^\infty_tH^1}\|u\|_{L^4([t_1,t_2],L^4)}^2\\
	&\lesssim \big(M(u_0)+E(u_0)\big)\|u\|_{L^4([t_1,t_2],L^4)}^2
\end{align*}
where, by virtue of Proposition \ref{prop:intMorest}, $\|u\|_{L^4([t_1,t_2],L^4)}^2\to 0$ as $t_1,t_2\to \infty$. 
Hence, we obtain \eqref{equ:1-chismall} for some $u_+\in L^2(\R^3)$, which completes the proof.

\appendix
\section{Proof of Lemmas \ref{lemma_CGT_1} and \ref{lemma_CGT_2}}\label{appendix_semiclassical}

This appendix is devoted to the proofs of Lemmas \ref{lemma_CGT_1} and \ref{lemma_CGT_2}. Although the strategy of the proofs follows the same line as that of \cite{BGT} (see also \cite{Bouclet}), we give the details of proofs as the magnetic potential $A(x)$ was not treated in those works. 

We first state basic $L^p\to L^q$ estimates for $h$-PDOs. 

%{lemma}
\begin{lemma}
	\label{lemma_appendix_1}
	Let $m>3$ and $1\le p\le q\le \infty$. Then there exists $C=C(p,q,m)>0$ such that, for any $b\in S^{-m}(\R^3)$
	$$
	\|b(x,hD)\|_{L^p(\R^3)\to L^q(\R^3)}\le C h^{-3(\frac1p-\frac1q)}\sup_{|\beta|\le 4}\|\<\xi\>^m\partial_\xi^\beta b\|_{L^\infty(\R^3\times\R^3)}. 
	$$
\end{lemma}

%proof
\begin{proof}
	Since $m>3$ and hence $b\in L^\infty(\R^3_x,L^1(\R^3_\xi))$, we can write
	$$
	b(x, hD)f(x)=\int_{\R^3}K(x,y,h)f(y)\,dy
	$$
	with
	$$
	K(x,y,h)=\frac{1}{(2\pi h)^3}\int_{\R^3}e^{i\frac{(x-y)}{h}\cdot\xi}b(x,\xi)d\xi=\frac{1}{(\sqrt{2\pi} h)^3}\hat{b}\Big(x,\frac{x-y}{h}\Big),
	$$
	where $\hat b$ is the Fourier transform of $b$ in  the second variable. In particular,
	$$
	\|b(x,hD)\|_{L^1\to L^\infty}\lesssim h^{-3}\|\<\xi\>^m b\|_{L^\infty}.
	$$
	Moreover, since
	\begin{align*}
		\int_{\R^3}|K(x,y,h)|dx&=\frac{1}{(2\pi)^{3/2}}\int_{\R^3}|\hat b(y+zh,z)|dz\lesssim \sup_{|\beta|\le4}\|\partial_\xi^\beta b\|_{L^\infty},\\
		\int_{\R^3}|K(x,y,h)|dy&=\frac{1}{(2\pi)^{3/2}}\int_{\R^3}|\hat b(x,y)|dy\lesssim \sup_{|\beta|\le4}\|\partial_\xi^\beta b\|_{L^\infty},
	\end{align*}
	the Schur lemma shows that for all $1\le p\le \infty$ $$\|b(x,hD)\|_{L^p\to L^p}\lesssim \sup_{|\beta|\le4}\|\partial_\xi^\beta b\|_{L^\infty}.$$
	The desired result thus follows by interpolation. 
\end{proof}

We next construct the approximation of $\varphi(-h^2\Delta_G)$ by $h$-PDOs.

%proposition
\begin{proposition}
	\label{proposition_appendix_2}
	Let $\varphi\in C_c^\infty(\R)$ and $1\le p\le q\le \infty$. Then there exists $N(p,q)\ge0$ such that, for any $N\ge N(p,q)$, there exist symbols $a_0,...,a_N\in S^{-\infty}$ satisfying $\supp a_j\subset \supp \varphi\circ p_0$ such that
	\begin{align}
		\label{proposition_appendix_2_1}
		\Big\|\varphi(-h^2\Delta_G)-\sum_{j=0}^Nh^j a_{j}(x,hD)\Big\|_{L^p\to L^q}\le C_{N,p,q}h^{N-N(p,q)},
	\end{align}
	uniformly in $h\in (0,1]$, where $p_0(x,\xi)=G(x)\xi\cdot\xi$ is the principal symbol of $-\Delta_G$. 
\end{proposition}

\begin{proof}
	%Although the proof follows a similar line as in \cite[Proposition 2.5]{Bouclet}, we give its detail for the sake of completeness.  
	We begin with Helffer-Sj\"ostrand's formula (see \cite[Section 8]{DiSj}): 
	\begin{align*}
		\varphi(-h^2\Delta_G)=-\frac{1}{2\pi i}\int_{\C}\partial_{\overline z}\widetilde\varphi(z)(-h^2\Delta_G-z)^{-1}d z d\overline z,
	\end{align*}
	where $dz\,d\overline z=-i2dxdy$ for $z=x+iy$ and $\widetilde \varphi$ is an almost analytic extension of $\varphi$ satisfying $\widetilde \varphi\in C_c^\infty(\R^2)$ with the identification $\C\cong \R^2$, $\widetilde \varphi(x)= \varphi(x)$ for all $x\in \R$ and $|\partial_{\overline z}\widetilde \varphi(z)|\le C_M|\Im z|^M$ for all $M\ge0$. %Note that the integral converges in the operator norm topology on $L^2$since$$\int_\C|\partial_{\overline z}\widetilde \varphi(z)\|(-h^2\Delta_G-z)^{-1}\|_{L^2\to L^2}|dzd\overline z|\le C_N\int_{\supp\widetilde \varphi}|\Im z|^{N-1}|dzd\overline z|<\infty$$
	To find symbols $a_j$ as in Proposition~\ref{proposition_appendix_2}, we construct the semiclassical parametrix of $(-h^2\Delta_G-z)^{-1}$. To this end, we write
	\begin{align}
		\label{proof_proposition_appendix_2_-1}
		-h^2\Delta_G%=(h\nabla-ihA)\cdot G(h\nabla-ihA)
		%&=-h^2\nabla\cdot G\nabla -2ih^2 A\cdot G\nabla-ih^2(\nabla\cdot GA)(x)+h^2(A\cdot GA)(x)\\
		=p_0(x,hD)+hp_1(x,hD)+h^2p_2(x),
	\end{align}
	where $p_0(x,\xi)=G(x)\xi\cdot\xi$ and $p_1(x,\xi)$ and $p_2(x)$ are the 1st order and potential terms of the symbol of $-h^2\Delta_G$, respectively. Note that $p_j\in S^{2-j}$ and $p_0$ satisfies 
	\begin{align}
		\label{proof_proposition_appendix_2_0}
		|p_0(x,\xi)-z|\gtrsim \<\xi\>^2|\Im z|.
	\end{align}
	By the composition rule of $h$-PDOs (see e.g. \cite{DiSj}),  one can find symbols $q_j\in S^{-2-j}$ for $j=0,...,N$ and $r_{N+1}\in S^{-3-N}$ such that
	$$
	(-h^2\Delta_G-z)\sum_{j=0}^Nh^jq_j(x,hD,z)=\Id+h^{N+1}r_{N+1}(x,hD,z,h).
	$$
	Indeed, $q_j=q_j(x,\xi,z)$ are defined inductively by
	\begin{align*}
		q_0=\frac{1}{p_0-z},\quad
		%q_1&=-\frac{1}{p_0-z}\Big\{\sum_{|\alpha|=1}i^{-|\alpha|}(\partial_\xi^\alpha q_0)(\partial_x^\alpha p_0)+q_0\cdot p_1\Big\},\\
		%q_2&=-\frac{1}{p_0-z}\Big\{\sum_{|\alpha|=2}i^{-|\alpha|}(\partial_\xi^\alpha q_0)(\partial_x^\alpha p_0)+\sum_{|\alpha|=1}i^{-|\alpha|}(\partial_\xi^\alpha q_1)(\partial_x^\alpha p_1)+q_0\cdot p_2\Big\},\\
		q_j=-\frac{1}{p_0-z}\sum_{\ell=0,1,2}\sum_{\substack{|\alpha|+k=j-\ell\\0\le k\le j-1}} i^{-|\alpha|}(\partial_\xi^\alpha q_k)(\partial_x^\alpha p_\ell).%+\sum_{|\alpha|+k=j-1\atop 0\le k\le j-1}i^{-|\alpha|}(\partial_\xi^\alpha q_k)(\partial_x^\alpha p_1)\\&\quad +\sum_{|\alpha|+j=k-2}i^{-|\alpha|}(\partial_\xi^\alpha q_j)(\partial_x^\alpha p_2)\Big\},\quad k\ge3. 
	\end{align*}
	By \eqref{proof_proposition_appendix_2_0}, $q_0$ satisfies
	$$
	|\partial_x^\alpha\partial_\xi^\beta q_0(x,\xi,z)|\le C_{\alpha\beta}\<\xi\>^{-2-|\beta|}|\Im z|^{-1-|\alpha|-|\beta|}.
	$$
	It also follows from an induction argument with respect to $N$ that
	\begin{align}
		\nonumber
		|\partial_x^\alpha\partial_\xi^\beta q_j(x,\xi,z)|&\le C_{\alpha\beta}\<\xi\>^{-2-j-|\beta|}|\Im z|^{-n_k-|\alpha|-|\beta|},\quad j=1,...,N,\\
		\label{proof_proposition_appendix_2_1}
		|\partial_x^\alpha\partial_\xi^\beta r_{N+1}(x,\xi,z,h)|&\le C_{\alpha\beta}\<\xi\>^{-3-N-|\beta|}|\Im z|^{-n_{N+1}-|\alpha|-|\beta|}
	\end{align}
	with some $n_1,...,n_{N+1}\in \mathbb N$,  and that, for each $j=1,...,N$,  $q_j$ is of the form
	$$
	q_j=\sum_{k=1}^{2j-1}\frac{\widetilde q_{jk}}{(p_0-z)^{1+k}}
	$$
	with some symbols $\widetilde q_{jk}=\widetilde q_{jk}(x,\xi)$ which are independent of $z,h$, polynomials in $\xi$ of order $2k-j$ and belong to $S^{2k-j}$. Hence, we have
	$$
	(-h^2\Delta_G-z)^{-1}=\sum_{j=0}^Nh^jq_j(x,hD,z)-h^{N+1}(-h^2\Delta_G-z)^{-1}r_{N+1}(x,hD,z,h).
	$$
	Plugging the identity 
	$$
	-\frac{1}{2\pi i}\int_{\C}\partial_{\overline z}\widetilde\varphi(z)(p_0-z)^{-1-k}d z d\overline z=\frac{(-1)^k}{k!}\varphi^{(k)}\circ p_0
	$$
	into Helffer-Sj\"ostrand's formula shows
	$$
	\varphi(-h^2\Delta_G)=\sum_{j=0}^Nh^j a_j(x,hD)+h^{N+1}R(h,N),
	$$
	where $a_0(x,\xi)=\varphi(p_0(x,\xi))$ and
	\begin{align*}
		a_j(x,\xi)&=\sum_{k=1}^{2j-1}\frac{(-1)^k}{k!}\widetilde q_{jk}(x,\xi)\varphi^{(k)}(p_0(x,\xi)),\quad j=1,...,N,\\
		R(h,N)&=\frac{1}{2\pi i}\int_{\C}\partial_{\overline z}\widetilde\varphi(z)(-h^2\Delta_G-z)^{-1} r_{N+1}(x,hD,z,h)d z d\overline z.
	\end{align*}
	By the ellipticity of $G$ and the condition $\varphi\in C_0^\infty(\R)$, we have $a_j\in S^{-\infty}$ and $\supp a_j\subset \supp \varphi\circ p_0\subset \{c_1\le |\xi|\le c_2\}$ with some $c_1,c_2>0$. In particular, 
	$$
	\|a_j(x,hD)\|_{L^p\to L^q}\lesssim h^{-3(\frac1p-\frac1q)}
	$$
	for any $1\le p\le q\le \infty$ uniformly in $h\in (0,1]$ by Lemma \ref{lemma_appendix_1}. 
	
	To deal with the term $R(h,N)$, we shall show that $(-h^2\Delta_G-z)^{-1}$ is an $h$-PDO. Define the multiple commutator operator $L^\alpha \widetilde L^\beta$ with $\alpha,\beta\in \Z_+^3 $ by
	$$
	L_{j}(H)=h^{-1}[x_j,H],\quad \widetilde L_{j}(H)=[\partial_{x_j},H],\quad L^\alpha \widetilde L^\beta (H)=L_1^{\alpha_1}\cdots L_d^{\alpha_d}\widetilde L_1^{\beta_1}\cdots \widetilde L_d^{\beta_d}(H)
	$$
	for a given operator $H$. A direct computation yields
	$$
	L_j((-h^2\Delta_G-z)^{-1})=(-h^2\Delta_G-z)^{-1}\cdot L_j(h^2\Delta_G)\cdot (-h^2\Delta_G-z)^{-1}.
	$$
	The same identity also holds with $L_j$ replaced by $\widetilde L_j$. By an induction, we find that $L^\alpha \widetilde L^\beta ((-h^2\Delta_G-z)^{-1})$ is a linear combination of
	$$
	(-h^2\Delta_G-z)^{-1}\Big(\prod_{k=1}^\ell L^{\alpha^{(k)}}\widetilde L^{\beta^{(k)}}(h^2\Delta_G)\Big) (-h^2\Delta_G-z)^{-\ell},
	$$
	where $\alpha^{(k)},\beta^{(k)}\in \Z_+^3$, $\alpha^{(k)},\beta^{(k)}\neq0$, $\alpha^{(1)}+\cdots+\alpha^{(\ell)}=\alpha$, $\beta^{(1)}+\cdots+\beta^{(\ell)}=\beta$ and $1\le \ell\le |\alpha|+|\beta|$. Since 
	\begin{align*}
		\|(-h^2\Delta_G-z)^{-1}L^{\alpha^{(k)}}\widetilde L^{\beta^{(k)}}(h^2\Delta_G)\|_{L^2\to L^2}&\le C_{\alpha^{(k)}\beta^{(k)}}\<z\>|\Im z|^{-1},\\
		\|(-h^2\Delta_G-z)^{-\ell}\<hD\>^2\>\|\le C_\ell \<z\>|\Im z|^{-\ell},
	\end{align*}
	this formula implies
	$$
	\|L^\alpha \widetilde L^\beta ((-h^2\Delta_G-z)^{-1})\<hD\>^2\|_{L^2\to L^2}\le C_{\alpha\beta} \<z\>^2|\Im z|^{-1-|\alpha|-|\beta|}. 
	$$
	Therefore, Beal's characterization of $h$-PDOs (see \cite[Proposition 8.3]{DiSj}) shows $$(-h^2\Delta_G-z)^{-1}=\widetilde r(x,hD,z,h)$$ with some symbol $\widetilde r(x,\xi,z,h)$ satisfying
	\begin{align}
		\label{proof_proposition_appendix_2_2}
		|\partial_x^\alpha \partial_\xi^\beta \widetilde r(x,\xi,z,h)|\le C_{\alpha\beta}\<\xi\>^{-2-|\beta|}\<z\>^2|\Im z|^{-1-|\alpha|-|\beta|}
	\end{align}
	uniformly in $x,\xi\in \R^3$, $z\in \supp \widetilde \varphi$ and $h\in (0,1]$. By the composition rule of $h$-PDOs and \eqref{proof_proposition_appendix_2_1} and \eqref{proof_proposition_appendix_2_2}, $\widetilde r(x,hD,z,h)r_{N+1}(x,hD,z,h)$ is an $h$-PDO with some symbol $s(x,\xi,z,h)$ satisfying
	$$
	|\partial_x^\alpha \partial_\xi^\beta s(x,\xi,z,h)|\le C_{\alpha\beta}\<\xi\>^{-5-N-|\beta|}\<z\>^2|\Im z|^{-1-n_{N+1}-|\alpha|-|\beta|}
	$$
	which, together with Lemma \ref{lemma_appendix_1}, shows
	$$
	\|s(x,hD,z,h)\|_{L^p\to L^q}\lesssim h^{-3(\frac1p-\frac1q)}\<z\>^2|\Im z|^{-5-n_{N+1}}. 
	$$
	Plugging this estimate into the formula of $R(h,N)$ and using the bound $|\partial_{\overline z}\widetilde \varphi(z)|\le C_M|\Im z|^M$ for $M\ge 5+n_{N+1}$ and the fact $\widetilde \varphi \in C_c^\infty(\R^2)$, we obtain \eqref{proposition_appendix_2_1}. 
\end{proof}

%Lemma \ref{lemma_CGT_1} clearly follows from Lemma \ref{lemma_appendix_1} and Proposition \ref{proposition_appendix_2}. 

%proof
\begin{proof}[Proof of Lemma \ref{lemma_CGT_1}]
	For the first estimate, choose $\psi_0\in C_c^\infty(\R)$ such that $\psi_0\equiv1$ on $\supp\varphi$. Then $\varphi(-h^2\Delta_G)=\psi_0(-h^2\Delta_G)\varphi(-h^2\Delta_G)$, and \eqref{lemma_CGT_1_1} follows from Lemma~\ref{lemma_appendix_1} and Proposition~\ref{proposition_appendix_2}. For the second estimate, the additional hypothesis $\supp\varphi\Subset(0,\infty)$ allows us to choose $\psi\in C_c^\infty((0,\infty))$ such that $\psi\equiv1$ on $\supp\varphi$. We then write
	$$
	\psi(-h^2\Delta_G)=\psi(-h^2\Delta_G)(-h^2\Delta_G)^{-s/2}\cdot (-h^2\Delta_G)^{s/2}.
	$$
	Since $\psi(\lambda)\lambda^{-s/2}\in C_c^\infty((0,\infty))$, \eqref{lemma_CGT_1_1} and Lemma \ref{lemma_2_2} imply \eqref{lemma_CGT_1_2} 
	%\begin{align*}\|\psi(-h^2\Delta_G)f\|_{L^q}&\lesssim h^{-3(1/p-1/q)-s}\|(-\Delta_G)^{s/2}f\|_{L^p}\\&\lesssim h^{-3(1/p-1/q)-s}\|f\|_{W^{s,p}_G}\\&\lesssim h^{-3(1/p-1/q)-s}\|f\|_{W^{s,p}}\end{align*}
	if  $p>1$ and $s\ge0$. \end{proof}

%proof
\begin{proof}[Proof of Lemma \ref{lemma_CGT_2}]
	By replacing $t$ with $th$, it follows from Keel-Tao's theorem \cite{KeelTao}, Lemma \ref{lemma_CGT_1} and Proposition \ref{proposition_appendix_2} that it suffices to show 
	\begin{align}
		\label{proof_lemma_CGT_2_1}
		\|\varphi(-h^2\Delta_G)e^{-ith\Delta_G}b(x,hD)\|_{L^1\to L^\infty}\lesssim |th|^{-3/2},\quad 0<t<\alpha,\ h\in (0,1],
	\end{align}
	where $b\in S^{-\infty}$ supported in $\R^3\times B$ with a bounded ball $B$. 
	
	To this end, by replacing $t$ with $th$, we shall construct the semiclassical parametrix of $e^{-ith\Delta_G}b(x,hD)f(x)$ up to $|t|\le \alpha$ in the following form
	$$
	w_N(t,x)=\frac{1}{(2\pi h)^{3}}\int_{\R^3}e^{\frac ih(\Psi(t,x,\xi)-y\cdot \xi)}\sum_{j=0}^Nh^j b_j(t,x,\xi)f(y)dyd\xi.
	$$
	Recalling the formula \eqref{proof_proposition_appendix_2_-1}, we have as an operator 
	\begin{align}
		\label{proof_lemma_CGT_2_2}
		&e^{-\frac ih\Psi}(hi\partial_t+h^2\Delta_G)e^{\frac ih\Psi}\\
		\nonumber
		&=-\{\partial_t \Psi+p_0(x,\nabla_x\Psi)\}+ih\left(\partial_t+\mathcal X\cdot \nabla_x+\mathcal Y\right)+h^2\Delta_G,
	\end{align}
	where $\mathcal X$ is a vector field and $\mathcal Y$ is a potential given by
	\begin{align*}
		\mathcal X(t,x,\xi)&=(\nabla_\xi p)(x,\nabla_x\Psi)(t,x,\xi),\\
		\mathcal Y(t,x,\xi)&=(p_0(x,\nabla_x)\Psi)(t,x,\xi)-i(p_1(x,\nabla_x)\Psi)(t,x,\xi)
	\end{align*}
	Since $p_0$ is uniformly positive definite and has smooth bounded coefficients $G_{j,k}$, by the Hamilton-Jacobi theory, there exists a sufficiently small $\alpha>0$ and a smooth function $\Psi(t,x,\xi)$ defined in $[-\alpha,\alpha]\times \R^6$ satisfying the Hamilton-Jacobi equation
	\begin{align}
		\label{proof_lemma_CGT_2_3}
		\partial_t \Psi+p_0(x,\nabla_x\Psi)=0;\quad \Psi|_{t=0}=x\cdot\xi
	\end{align}
	for $(x,\xi)\in \R^3\times B$. Moreover, 
	\begin{align}
		\label{proof_lemma_CGT_2_4}
		|\partial_x^\alpha\partial_\xi^\beta(\Psi(t,x,\xi)-x\cdot \xi+tp_0(x,\xi))|\le C_{\alpha\beta}|t|^2
	\end{align}
	for $(x,\xi)\in \R^3\times B$ and $|t|\le \alpha$. Then we observe from the identities \eqref{proof_lemma_CGT_2_2} and \eqref{proof_lemma_CGT_2_3} that, in order that $w_N$ solves approximately the semiclassical Schr\"odinger equation
	$$
	(hi\partial_t+h^2\Delta_G)w_N=0;\quad w_N|_{t=0}=b(x,hD)f,
	$$
	the symbols $b_0,...,b_N$ should satisfy the transport equations
	\[
		\left\{\begin{aligned}
			&\partial_t b_0+\mathcal X\cdot\nabla_xb_0+\mathcal Yb_0=0;\quad b_0|_{t=0}=b,\\
			&\partial_t b_j+\mathcal X\cdot\nabla_xb_j+\mathcal Yb_j-i\Delta_Gb_j=0;\quad b_j|_{t=0}=0,\quad j=1,\ldots,N.
		\end{aligned}
		\right.
	\]
	Since $\mathcal X$ and $\mathcal Y$ are smooth and all of their derivatives are bounded on $[-\alpha,\alpha]\times \R^3\times B$, one can in fact construct such symbols $b_0,...,b_N\in S^{-\infty}$ by taking $\alpha$ small further if necessary so that they satisfy, for any $M\ge0$, 
	\begin{align}
		\label{proof_lemma_CGT_2_5}
		|\partial_x^\alpha\partial_\xi^\beta b_j(t,x,\xi)|\le C_{\alpha\beta M}\<\xi\>^{-M},\quad j=1,...,N,
	\end{align}
	on $[-\alpha,\alpha]\times \R^3\times B$. Therefore, we have the following Duhamel formula
	$$
	e^{-ith\Delta_G}b(x,hD)f=w_N-ih^N\int_0^te^{-i(t-s)h\Delta_G}r_{N+1}(s,x,hD)fds
	$$
	with some symbol $r_{N+1}$ satisfying the same estimates as \eqref{proof_lemma_CGT_2_5}. 
	
	Now, by \eqref{proof_lemma_CGT_2_4} and the ellipticity of $G$, we can apply the stationary phase theorem and Lemma \ref{lemma_CGT_1} to obtain
	$$
	\|\varphi(-h^2\Delta_G)w_N\|_{L^\infty}\lesssim \|w_N\|_{L^\infty}\lesssim |th|^{-\frac 32}
	$$
	for $0<t<\alpha$ and $h\in (0,1]$. To estimate the remainder term, we use Lemmas \ref{lemma_CGT_1} and \ref{lemma_appendix_1} to observe that
	\begin{align*}
		&\Big\|h^N\varphi(-h^2\Delta_G)\int_0^te^{-i(t-s)h\Delta_G}r_{N+1}(s,x,hD)fds\Big\|_{L^\infty}\\
		&\lesssim h^{N-\frac32}\int_0^t \|r_{N+1}(s,x,hD)f\|_{L^2}ds\\
		&\lesssim h^{N-3}|t|\|f\|_{L^1},
	\end{align*}
	which completes the proof of \eqref{proof_lemma_CGT_2_1} and thus the lemma. 
\end{proof}
\section{A complementary abstract result for constant damping}\label{appendix:abstract}
	This appendix records a complementary result for nonlinear Schr\"odinger equations with constant damping and a general positive Hamiltonian. The result is logically independent of Theorems~\ref{T:gwp} and~\ref{T:main}: it relies on global Strichartz estimates and therefore does not cover the trapping metrics treated in the main body. 

	The proof follows Inui's strategy and is given below. However, compared to Inui's result,  we can provide several important examples including the fractional NLS, magnetic NLS and second-order elliptic operator with divergence form. 
	\begin{example}[Fractional Schr\"odinger operator]
		Let $H=(-\Delta)^s$ with $0<s\le 1$, so $q(u)=\|(-\Delta)^{s/2}u\|_2^2$ and $D(H^{1/2})=H^s(\mathbb R^d)$.
		The energy-subcritical condition is
		\[
		1<p<1+\frac{4s}{d-2s}\quad (d>2s),
		\qquad \text{and all }p>1\text{ if }d\le 2s.
		\]
		The fractional Gagliardo--Nirenberg inequality reads
		\[
		\|u\|_{L^{p+1}}^{p+1}
		\le
		C_{GN}\,\|(-\Delta)^{s/2}u\|_{2}^{\frac{d}{2s}(p-1)}\,\|u\|_2^{(p+1)-\frac{d}{2s}(p-1)},
		\]
		where the sharp constant $$C_{GN}=\frac{2s(p+1)}{d(p-1)}\|Q\|_{L^2}^{\frac{d(p-1)}{2s}-p-1}\|Q\|_{\dot H^s}^{2-\frac{d(p-1)}{2s}}$$
		and $Q$ is the static solution to the elliptic equation
		\begin{align*}
			(-\Delta)^sQ+Q=|Q|^{p-1}Q.
		\end{align*}
		
	\end{example}
	
	For the magnetic Schr\"odinger operator $H_A=-\Delta_A+V(x)=(-i\nabla-A)^2+V(x)$, Fanelli and Vega \cite{Fanelli-Vega} established magnetic virial identities for both wave and Schr\"odinger equations. Under some assumptions on the electric and magnetic potentials $V,A$, they also obtain some weak dispersive estimates and Strichartz estimates for a restricted range of admissible pairs. Later, D'Ancona-Fanelli-Vega-Visciglia \cite{DFVV-endpoint} proved the full range of Strichartz estimate which includes the endpoint case in $d=3$. Colliander-Czubak-Lee \cite{CCL} proved the interaction Morawetz estimate under the assumption of \cite{Fanelli-Vega} and additional assumption of $DA=(\partial_jA_i)$ and $\nabla V$.
	\begin{example}[Magnetic Schr\"odinger operator]Let $A\in (L_{\operatorname{loc}}^2(\R^d))^d$, then
		$H=-\Delta_A=(-i\nabla-A(x))^2$ be defined as the Friedrichs operator associated with
		\[
				q_A(u)=\|(\nabla-iA)u\|_2^2,\qquad D(q_A)=\{u\in L^2(\mathbb R^d):\nabla_Au\in L^2(\mathbb R^d)\},
				\]
		for a real magnetic potential $ A$ such that the form is closed and nonnegative.
		The key additional ingredient is the diamagnetic inequality:
		\begin{equation}\label{eq:diamag-1}
			|\nabla |u||\ \le\ |(\nabla-i A)u|\qquad \text{a.e. }x\in\R^d.
		\end{equation}
		Consequently, applying classical G-N inequality to $|u|$ and using diamagnetic inequality \eqref{eq:diamag-1} gives
		\[
		\|u\|_{L^{p+1}}^{p+1}
		\le
		C_{\mathrm{GN}}\,
		\|(\nabla-i A)u\|_2^{\frac{d}{2}(p-1)}\|u\|_2^{(p+1)-\frac{d}{2}(p-1)}.
		\]
	\end{example}
	\begin{remark}
				The magnetic quadratic form associated with $(\nabla - iA)^2$ is well-defined
				and closed under suitable assumptions on $A$, as shown by Leinfelder-Simader \cite{Leinfelder}.
				In particular, if $A \in (L^2_{\mathrm{loc}}(\mathbb{R}^d))^d$, the form is closed
				on the natural domain
				\[
				\{u \in L^2 : (\nabla - iA)u \in L^2\}.
				\]
	\end{remark}

\begin{example}[Divergence-form elliptic operators]
Let
		\[
		H=-\mathrm{div}(G(x)\nabla)
		\]
		on $L^2(\mathbb R^d)$, where $G(x)$ is symmetric and uniformly elliptic:
		there exist $0<\lambda\le\Lambda$ such that
		\[
		\lambda|\xi|^2\le \xi^{\mathsf T}G(x)\xi\le \Lambda|\xi|^2\quad\forall \xi,\ \text{a.e. }x.
		\]
		Define $H$ by Friedrichs from
		\[
		q_G(u)=\int_{\mathbb R^d}\nabla u(x)^{\mathsf T}G(x)\nabla\overline{u(x)}\,dx,\qquad D(q_G)=H^1(\mathbb R^d).
		\]
		Then we have $q_G(u)\simeq \|\nabla u\|_2^2$. It implies that the  Gagliardo-Nirenberg inequality similar to \eqref{eq:GN-abstract} holds. To apply the abstract $H^1$ scattering framework, one additionally needs global-in-time Strichartz estimates.
 Staffilani-Tataru \cite{Staffilani-Tataru} and Burq-G\'erard-Tzvetkov \cite{BGT} proved the local-in-time Strichartz estimate with loss of derivatives, while this loss will be removed if assuming the metric $G$ is non-trapping. We also refer to Bouclet-Tzvetkov \cite{Bouclet} for the case of long range perturbation of flat metric. Later, Marzuola-Metcalfe-Tataru \cite{MMT-JFA} proved the global-in-time Strichartz estimate for asymptotic Euclidean metric under non-trapping assumption in $d\geq3$ and its conical extension was proved by Hassell-Zhang \cite{Hassell-Zhang}.
	\end{example}
	Next, we provide another example which combines the previous two settings: magnetic Laplacian and variable coefficients.
	\begin{example}
		Let $G:\mathbb R^d\to\mathbb R^{d\times d}$ be a measurable symmetric matrix  satisfying the
		uniform ellipticity and boundedness: there exist $0<\lambda\le \Lambda<\infty$ such that
		\begin{equation}\label{eq:ellipticity-magnetic}
			\lambda |\xi|^2 \le \xi^{\mathsf T}G(x)\,\xi \le \Lambda |\xi|^2
			\qquad \forall \xi\in\mathbb R^d,\ \text{a.e. }x\in\mathbb R^d.
		\end{equation}
		Let $A:\mathbb R^d\to\mathbb R^d$ be a real-valued magnetic potential such that the form below is
		well-defined and closed (for instance $A\in L^2_{\mathrm{loc}}(\mathbb R^d)$).
		
		Define the covariant gradient $\nabla_{ A}:=\nabla-i A$ and consider the quadratic form
		\begin{equation}\label{eq:qGA}
			q_{G, A}(u)
			:=
			\int_{\mathbb R^d} \big(\nabla_{A}u(x)\big)^{\mathsf T}
			G(x)\,
			\overline{\nabla_{ A}u(x)}\ dx,
			\quad  D(q_{G,A}) = \{ u \in L^2(\mathbb{R}^d) : \nabla_A u \in L^2(\mathbb{R}^d)\}.
		\end{equation}
		By \eqref{eq:ellipticity-magnetic}, $q_{G,A}$ is nonnegative and equivalent to the magnetic form norm $\|\nabla_Au\|_2^2$ on its natural domain. Hence it is closed and defines a nonnegative selfadjoint operator $H\ge0$ via Friedrichs extension:
		\[
		H=-\nabla_A\cdot G(x)\nabla_A
		\quad\text{in the sense that}\quad
		\langle Hu,v\rangle = q_{G, A}(u,v)\ \ \ (u\in D(H),\ v\in D(q_{G, A})).
		\]
       
The associated energy is
\[
E(u)=\frac12\,q_{G,A}(u)-\frac{\mu}{p+1}\|u\|_{L^{p+1}}^{p+1}.
\]

Using the diamagnetic inequality
\begin{equation}\label{eq:diamag-GA}
	|\nabla |u||\ \le\ |\nabla_{ A}u|\qquad \text{a.e.}x\in\R^d,
\end{equation}
and \eqref{eq:diamag-GA},
\[
\|\nabla |u|\|_2^2 \le \|\nabla_{ A}u\|_2^2 \lesssim q_{G, A}(u),
\]
and applying the classical Gagliardo--Nirenberg inequality to $|u|$ yields, in the energy-subcritical range
\begin{align*}
	\|u\|_{p+1}^{p+1}
	\leq
	C_{\mathrm{GN}}\,
	q_{G,A}(u)^{\frac{d}{4}(p-1)}\|u\|_2^{(p+1)-\frac{d}{2}(p-1)},
\end{align*}
where
\[
1<p<1+\frac{4}{d-2}\quad(d>2),\qquad \text{and all }p>1\text{ if }d= 2.
\]
		Let $G$ be a non-trapping metric, assume that there exist $\mu>0$, for any $\alpha\in \Bbb Z_+^d$ such that \begin{align*}
    |\partial_x^\alpha (g^{jk}(x) - \delta_{jk})| &\leq C_\alpha \langle x \rangle^{-\mu - |\alpha|}, \\
    |\partial_x^\alpha A_j(x)| &\leq C_\alpha \langle x \rangle^{1 - \mu - |\alpha|}, \\
    |\partial_x^\alpha V(x)| &\leq C_\alpha \langle x \rangle^{2 - \mu - |\alpha|}, \quad x \in \mathbb{R}^d,
\end{align*} Mizutani \cite{Miz13} showed the local-in-time Strichartz estimate without loss of derivatives. Under the stronger smallness condition
\begin{align*}
  |G-I|+\langle x\rangle(|\nabla G|+|A|)+\langle x\rangle^2(|\nabla^2G|+|\nabla A|)\leq \varepsilon\langle x\rangle^{-\delta},\,\,\mbox{for some }\varepsilon,\delta>0,
\end{align*}
Cassano-D'Ancona \cite{CD16} proved the global-in-time Strichartz estimate.

	\end{example}
	\begin{remark}
		Under the uniform ellipticity of $G$,  we can show that $q_{A,G}(u)\simeq \|\nabla_Au\|_{L^2}^2$, which is useful in the focusing coercivity argument.
	\end{remark}

\subsection{Mass and energy decay under the damping flow}
\begin{lemma}[Mass dissipation identity]\label{lem:mass}
	Assume Assumption \ref{ass:damp}. Let $u$ be a sufficiently regular solution to \eqref{eq:damped}. Then
	\begin{equation}\label{eq:mass-id}
		\frac{d}{dt}\norm{u(t)}_{\mathcal H}^2 \;=\; -2\,\ip{\mathcal B u(t)}{u(t)}.
	\end{equation}
	In particular, $\norm{u(t)}_{\mathcal H}$ is nonincreasing in time.
	If  $a>0$, then
	\begin{equation}\label{eq:mass-exp}
		\norm{u(t)}_{\mathcal H}^2 \;\le\; e^{-2at}\,\norm{u(0)}_{\mathcal H}^2
		\qquad \forall t\ge 0.
	\end{equation}
\end{lemma}

\begin{proof}
	Differentiate $\norm{u(t)}_{\mathcal H}^2=\ip{u(t)}{u(t)}$:
	\[
	\frac{d}{dt}\norm{u(t)}_{\mathcal H}^2
	= 2\,\Ree\,\ip{\partial_t u(t)}{u(t)}.
	\]
	From \eqref{eq:damped},
	\[
	\partial_t u = -i(Hu+\mu |u|^{p-1}u) - \mathcal  Bu.
	\]
	Hence
	\[
	2\,\Ree\,\ip{\partial_t u}{u}
	=2\,\Ree\Big(\ip{-i(Hu+\mu |u|^{p-1}u)}{u}-\ip{\mathcal Bu}{u}\Big).
	\]
	The first term has zero real part since $\ip{Hu}{u}\in\mathbb R$  and
	\[
	\ip{|u|^{p-1}u}{u}=\int_{\R^d} |u|^{p+1}\,dx \in \mathbb R,
	\]
	independently of the sign of $\mu$.
	Thus we obtain \eqref{eq:mass-id}.
	Then Gr\"onwall's inequality and the lower bound of damping operator $\mathcal{B}$ yield \eqref{eq:mass-exp}.
\end{proof}

Next, we aim to establish the global well-posedness on $D(H^\frac{1}{2})$. By Assumption \ref{ass-strichartz}, $TT^*$ argument and Christ-Kiselev's lemma, we obtain the following inhomogeneous Strichartz estimate.
\begin{lemma}[Inhomogeneous Strichartz estimate] Let $(q,r)$ be the same as Assumption \ref{ass-strichartz}, then it holds
	\begin{align*}
		\Big\|\int_0^te^{i(t-s)H}F(s)\,ds\Big\|_{L_t^qH_x^{s,r}}\leq C\|F\|_{L_{t}^{q^\prime}H_x^{s,r^\prime}},
	\end{align*}
	where $q^\prime$ and $r^\prime$ are conjugate numbers of $q$ and $r$ respectively.
\end{lemma}
Combining the Strichartz estimate and Picard iteration theorem, we have the following local well-posedness.
\begin{lemma}[Local well-posedness]\label{LWP}
	
	Let $\mu=\pm1$ and $1<p<1+\frac{4s}{d-2s}$ where $s$ is corresponding to the leading order of $H$, then there exists $T_{\operatorname{max}}=T_{\operatorname{max}}(\|u_0\|_{D(H^\frac{1}{2})})\in(0,\infty]$ and \eqref{eq:damped} has the unique local solution $u\in C(I,D(H^\frac{1}{2}))\cap L_t^{q}(I,H_x^{s,r})$ where $(q,r)$ is given by Assumption \ref{ass-strichartz} and $I=[0,T_{\operatorname{max}})$. If    $T_{\operatorname{max}}<\infty$, it holds $\lim\limits_{t\to T_{\operatorname{max}}}\|u(t)\|_{D(H^\frac{1}{2})}=\infty$.
\end{lemma}
Since $\mathcal{B}$ is not trivial, i.e. $\mathcal{B}\neq0$, \eqref{eq:damped} is not a Hamiltonian system. To control the energy, we shall prove the following energy decay estimate under some assumption. 
Define the energy on $D(H^{1/2})$ by
\begin{equation}\label{eq:energy}
	E(u)
	=
	\frac12\,q(u)
	-
	\frac{\mu}{p+1}\,\norm{u}_{L^{p+1}}^{p+1},
\end{equation}
where $q(u)=\|H^{1/2}u\|_{L^2}^2$. 
\begin{lemma}[Abstract energy dissipation identity]
	Assume Assumption \ref{ass:damp}. Let $u$ be a sufficiently regular solution to \eqref{eq:damped}. Then
	\begin{equation}\label{eq:energy-dissipation}
		\frac{d}{dt}E(u(t))
		=
		-\Ree\,\ip{Hu(t)-\mu |u(t)|^{p-1}u(t)}{\mathcal B u(t)}.
	\end{equation}
\end{lemma}

\begin{proof}
	As in standard NLS computations,
	\begin{gather*}
		\frac{d}{dt}\Big(\tfrac12 q(u(t))\Big)
		=
		\Ree\,\ip{Hu(t)}{\partial_t u(t)},\\
		\qquad
		\frac{d}{dt}\Big(\tfrac{\mu}{p+1}\|u(t)\|_{L^{p+1}}^{p+1}\Big)
		=
		\Ree\,\ip{\mu |u(t)|^{p-1}u(t)}{\partial_t u(t)}.
	\end{gather*}
	Thus
	\[
	\frac{d}{dt}E(u(t))
	=
	\Ree\,\ip{Hu(t)-\mu |u(t)|^{p-1}u(t)}{\partial_t u(t)}.
	\]
	Using $\partial_t u=-i(Hu+\mu |u|^{p-1}u)-\mathcal Bu$, the term involving
	$-i(Hu+\mu |u|^{p-1}u)$ has zero real part, hence we have \eqref{eq:energy-dissipation}.
\end{proof}
%It is easy to show that  the energy is coercive along the flow in the constant damping case, i.e.
%%\begin{assumption}[Coercive dissipation along the flow]\label{ass:coercive-diss}
%there exists $c_0>0$ such that for all $t\ge0$ along the considered solution,
%\begin{equation}\label{eq:coercive-diss}
%\Ree\,\ip{Hu(t)-\mu |u(t)|^{p-1}u(t)}{\mathcal B u(t)}
%\ \ge\ c_0\,E(u(t)).
%\end{equation}
%%\end{assumption}
%
%\begin{assumption}[Energy coercivity along the flow]\label{ass:energy-coercive}
%There exists a constant $c_E>0$ such that, on the considered solution class,
%the energy controls the energy norm in the sense that
%\begin{equation}\label{eq:energy-coercive}
%E(u(t)) \ \ge\ c_E\,\|u(t)\|_{D(H^{1/2})}^2
%\qquad \forall t\ge 0.
%\end{equation}
%\end{assumption}
%

\subsection{Proof of Theorem \ref{thm-abstract}}

For the defocusing case,  the energy decay holds automatically. Then combining with the local well-posedness, we have the  global well-posedness in the full sub-critical range $1<p<1+\frac{4}{d-2s}$. For convenience, we state the theorem below.

\begin{proposition}[Global well-posedness in the defocusing case]
		Let $\mu=1$, and let $u$ be the unique maximal solution given by Theorem \ref{LWP}.
		By the energy decay and the coercivity of the energy, the maximal solution satisfies
		\[
		\sup_{t\in[0,T_{\max})}\|u(t)\|_{D(H^{1/2})}<\infty.
		\]
		Hence
		\[
		T_{\max}=+\infty.
		\]
		In particular, the solution is global forward in time.
\end{proposition}

However, for the focusing case, since the sign of kinetic energy and potential energy is not same, we need a further restriction on the size of initial data.  To bound the potential energy in this case, we shall use the sharp Gagliardo-Nirenberg inequality. Following the strategy in Weinstein \cite{Weinstein}, there exists a threshold $m_c$ such that when $\|u_0\|_{L^2}<m_c$ where $m_c$ is associated with the sharp constant in the Gagliardo-Nirenberg inequality, the solution exists globally in the $L^2$-subcrtical case. 

With global well-posedness in hand, we now turn to prove the exponentially scattering in the sense of Definition \ref{def:exp-scatt}. We begin with a scattering criterion.
\begin{lemma}[Scattering criterion] Let $\mu=\pm1$ and $1<p<1+\frac{4s}{d-2s}$. Assume that there exists a constant $B>0$ such that $\|e^{\mathcal Bt}u\|_{D(H^\frac{1}{2})}<B$, then $u$ exponentially scatters in $D(H^\frac{1}{2})$.
	
\end{lemma}
\begin{proof}
	First, we show that $\|e^{\mathcal Bt}u\|_{D(H^\frac{1}{2})}<B$ implies the global-in-time $L_t^{q_1} L_x^r$ bound. Indeed, using the Sobolev embedding, mass-energy  decay (cf. \eqref{eq:mass-exp} and \eqref{eq:energy-exp-decay}), we obtain
	\begin{align*}
		\|u\|_{L_t^{q_1}L_x^r([0,\infty),\R^d)}\leq C\|e^{-at}\|e^{at}u\|_{D(H^\frac12)}\|_{L_t^{q_1}}\leq C\|e^{-at}\|_{L_t^{q_1}([0,\infty))}<\infty.
	\end{align*}
	Next, we show that the global-in-time bound  implies the weighted Strichartz bound.  Then for $0<t_1<t_2<\infty$, it follows from the Duhamel formula and H\"older's inequality,
	\begin{align*}
		&\hspace{3ex}\|\<H\>^\frac12e^{\mathcal Bt}u\|_{ L_t^q((t_1,t_2),L^r(\R^d))}\\&\lesssim\|\<H\>^\frac12e^{-itH}u(t_1)\|_{L_t^q((t_1,t_2),L_x^r)}+\Big\|\int_{t_1}^{t_2}e^{-i(t-s)H}\<H\>^\frac12\big(e^{\mathcal Bs}|u(s)|^{p-1}u(s)\big)\,ds\Big\|_{L_t^q((t_1,t_2),L_x^r)}\\
		&\lesssim\|u(t_1)\|_{D(H^\frac12)}+\|u\|_{L_t^{q_1}L_x^r}^{p-1}\|\<H\>^\frac12e^{\mathcal Bt}u\|_{L_t^qL_x^r},
	\end{align*}
	where $(q,r)=(\frac{4s(p+1)}{d(p-1)},p+1)$ and  $1-\frac{2}{q}=\frac{p-1}{q_1}$. Indeed, we have $q_1=\frac{(p-1)q}{q-2}>0$.
	Taking $t_1$ sufficiently large, we have $\|u\|_{L_t^qL_x^r((t_1,t_2)\times\R^d)}\ll1$, which yield the following estimate by bootstrap argument
	\begin{align*}
		\|\<H\>^\frac12e^{\mathcal Bt}u\|_{L_t^{q_1} L_x^r}\leq C\|u(t_1)\|_{D(H^\frac{1}{2})}.
	\end{align*}
	For the fixed $t_1>0$, we can take $t_2\to\infty$, which yield the estimate in the interval $t\in(t_1,\infty)$. This is acceptable. Now, we turn to prove that $u$ exponentially scatters. By definition, we only need to prove that the scattering state $u_+\in D(H^\frac12)$. Taking $0<t<\tau$, by fractional Leibniz's rule and Duhamel's formula, we deduce 
	\begin{align*}
		&\quad\Big \|e^{\mathcal Bt}e^{itH}u(t)-e^{\mathcal B\tau}e^{i\tau H}u(\tau)\Big\|_{D(H^\frac{1}{2})}\\
		&\lesssim\Big\|\int_t^\tau e^{isH}\big(e^{\mathcal Bs}|u(s)|^{p-1}u(s)\big)\,ds\Big\|_{D(H^\frac{1}{2})}\\
		&\lesssim\|u\|_{L_t^{q_1}((t,\tau),L^r)}^{p-1}\|\<H\>^\frac12e^{\mathcal Bt}u\|_{L_t^q((t,\tau),L^r)}.
	\end{align*}  
	Since $D(H^\frac{1}{2})$ is a Hilbert space and then $e^{\mathcal Bt}e^{itH}u$ converges to $u_+$ in $D(H^\frac12)$.
\end{proof}
Next, we begin the  proof of Theorem \ref{thm-abstract}.
\begin{proof}[Proof of Theorem \ref{thm-abstract}]
	In the following, we only need to find a constant $B>0$ such that $\|e^{\mathcal Bt}H^\frac12u\|_{L^2}<B$ for examples presented in the introduction. In our setting, there exists $a>0$ such that $\mathcal{B}=bI$ with $b\geq a$. First, we define the functional $H(t)$ by
	\begin{align*}
		H(u(t)) &:= \frac{1}{2} e^{2\mathcal Bt} q(u) - \mu \frac{1}{p+1} e^{2\mathcal Bt} \left\| u(t) \right\|_{L^{p+1}}^{p+1} \\
		&\quad -\mu \frac{p-1}{p+1} \int_0^t e^{2\mathcal Bs} \left\| u(s) \right\|_{L^{p+1}}^{p+1} ds \\
		&= e^{2\mathcal Bt} E(u(t))- \mu \frac{p-1}{p+1} \int_0^t e^{2\mathcal Bs} \left\| u(s) \right\|_{L^{p+1}}^{p+1} ds.
	\end{align*}
	By the direct calculation, we can obtain $\frac{d}{dt}H(t)=0$. When $\mu=+1$ (i.e. defocusing), we have
	\begin{align*}
		e^{\mathcal Bt}q(u)^\frac12\leq \sqrt{2H(t)}<\infty,
	\end{align*}
	which completes the proof of Theorem \ref{thm-abstract} in the defocusing case. For the focusing case, we will use the  the Gagliardo-Nirenberg inequality. For simplicity, we only consider the case $s=1$ while other cases can be calculated similarly. In the following, we obtain the lower bound of $H(t)$
	\begin{equation}
		\begin{aligned}
			H(u(t)) &\geq \frac{1}{2} e^{2\mathcal Bt}q(u) 
			- \frac{C}{p+1} e^{2\mathcal Bt} \left\| u(t) \right\|_{L^2}^{p+1-\frac{d}{2s}(p-1)} 
			q(u)^{\frac{d(p-1)}{4s}} \\
			&\quad - C \frac{p-1}{p+1} \int_0^t e^{2\mathcal B\tau} \left\| u(s) \right\|_{L^2}^{p+1-\frac{d}{2s}(p-1)} 
			q(u)^{\frac{d(p-1)}{4s}}\, d\tau \\
			&\geq \frac{1}{2} \left(e^{\mathcal Bt}q(u)^\frac12\right)^2 
			- \frac{C}{p+1} e^{-b(p-1)\mathcal Bt} \left\| u_0 \right\|_{L^2}^{p+1-\frac{d}{2s}(p-1)} 
			\left(e^{\mathcal Bt}q(u)^\frac12\right)^{\frac{d(p-1)}{2s}} \\
			&\quad - C \frac{p-1}{p+1} \left\| u_0 \right\|_{L^2}^{p+1-\frac{d}{2s}(p-1)} 
			\int_0^t e^{-b(p-1)\mathcal B\tau} \left(e^{\mathcal B\tau} q(u)^\frac12\right)^{\frac{d(p-1)}{2}} ds.\label{eq:H}
		\end{aligned}
	\end{equation}
	In the last inequality, we need the following relationship on parameter $b$:
	\begin{align*}
		-(p-1)b+\frac{d}{2s}(p-1)>2,
	\end{align*}
	which is equivalent to
	\begin{align*}
		\frac{d}{2s}>b>\frac{d}{2s}-\frac{2}{p-1}>0.
	\end{align*}
	Since we are in the $L^2$-subcritical case, we can take $b=1$.
	
	For convenience, we let $f(t)=e^{\mathcal Bt}q(u)^\frac{1}{2}$. Then we rewrite
	\begin{gather*}
		H(u_0) \geq \frac{1}{2} f(t)^2 - \frac{2}{p+1} e^{-(p-1)\mathcal Bt} \| u_0 \|_{L^2}^{p+1-\frac{d}{2s}(p-1)} f\left(t\right)^{\frac{d(p-1)}{2s}} \notag \\
		- 2\frac{p-1}{p+1} \| u_0 \|_{L^2}^{p+1-\frac{d}{2s}(p-1)} \int_0^t e^{-(p-1)\mathcal B\tau} f(\tau)^{\frac{d(p-1)}{2s}} \,d\tau.
	\end{gather*}
	Then by the lower bound of damping operator and  Young's inequality yields 
	\begin{gather*}
		\frac{2}{p+1} e^{-(p-1)\mathcal Bt} \|u_0\|_{L^2}^{p+1-\frac{d}{2s}(p-1)} f(t)^{\frac{d(p-1)}{2s}} \notag \\
		\leq C(\|u_0\|_{L^2}) e^{-\frac{p-1}{2}at} \cdot e^{-\frac{p-1}{2}at} f(t)^{\frac{d(p-1)}{2s}} \notag \\
		\leq \frac{1}{4} e^{-\frac{2s}{d}at} f(t)^2 + C(\|u_0\|_{L^2}) e^{-\frac{2s(p-1)}{4s-d(p-1)} at}.
	\end{gather*}
	Substituting the above inequality in the lower bound, we get
	\begin{align}
		H(u_0) &\geq \left( \frac{1}{2} - \frac{1}{4}e^{-\frac{2s}{d}\mathcal Bt} \right)f(t)^2 - C(\|u_0\|_{L^2})e^{-\frac{2s(p-1)}{4s-d(p-1)}at} \nonumber \\
		&\quad -(p-1)\left\{ \int_0^t \frac{1}{4}e^{-\frac{2s}{d}a \tau}f(\tau)^2\,d\tau + C(\|u_0\|_{L^2})\int_0^t e^{-\frac{2s(p-1)}{4s-d(p-1)}a\tau}\,d\tau \right\} \nonumber \\
		&\geq \frac{1}{4}f(t)^2 - C(\|u_0\|_{L^2}) - \int_0^t \frac{p-1}{4}e^{-\frac{2s}{d}a\tau}f(\tau)^2\,d\tau.
	\end{align}
	since $p<1+\frac{4s}{d}$, we have \( e^{-\frac{2s(p-1)}{4s-d(p-1)}} \leq 1 \) and \( \int_0^t e^{-\frac{2s(p-1)}{4s-d(p-1)}a\tau}\,d\tau < C \), this implies that
	\begin{equation}
		\frac{1}{4}f(t)^2 \leq H(u_0) + C(\|u_0\|_{L^2}) + \int_0^t \frac{p-1}{4}e^{-\frac{2s}{d}a\tau}f(\tau)^2\,d\tau.
	\end{equation}
	Applying the Gronwall inequality, we get
	\begin{equation}
		f(t)^2 \leq 4(H(u_0) + C(\|u_0\|_{L^2})) \exp\left((p-1)\int_0^t e^{-\frac{2s}{d}a\tau}\,d\tau\right) \lesssim H(u_0) + C(\|u_0\|_{L^2}).
	\end{equation}
	This finishes the proof of Theorem \ref{thm-abstract}. As a consequence, we also prove the global existence for the $L^2$-subcritical case with focusing nonlinearity. 

Now, we consider the mass-critical focusing case
\[
\mu=-1,
\qquad
p=1+\frac{4s}{d}.
\]
Set
\[
Y(t):=e^{2b t}q(u(t)),
\qquad
\Theta
:=
\frac{2C_{\mathrm{GN}}}{p+1}
\|u_0\|_{L^2}^{p-1}.
\]
By the mass dissipation identity,
\[
\|u(t)\|_{L^2}
=
e^{-bt}\|u_0\|_{L^2}.
\]
Since
\[
\|u(t)\|_{L^{p+1}}^{p+1}
\le
C_{\mathrm{GN}}
q(u(t))
\|u(t)\|_{L^2}^{p-1},
\]
we obtain
\[
e^{2bt}\|u(t)\|_{L^{p+1}}^{p+1}
\le
C_{\mathrm{GN}}
\|u_0\|_{L^2}^{p-1}
e^{-b(p-1)t}Y(t).
\]

Using the weighted energy identity, we have
\begin{equation}
\label{equ:mass-critical-lower-bound}
E_-(u_0)
\ge
\frac{1-\Theta}{2}Y(t)
-
\frac{b(p-1)\Theta}{2}
\int_0^t
e^{-b(p-1)\tau}Y(\tau)\,d\tau.
\tag{B.13}
\end{equation}
Consequently, provided that $\Theta<1$,
\begin{equation}
\label{equ:mass-critical-gronwall}
Y(t)
\le
\frac{2E_-(u_0)}{1-\Theta}
+
\frac{b(p-1)\Theta}{1-\Theta}
\int_0^t
e^{-b(p-1)\tau}Y(\tau)\,d\tau.
\tag{B.14}
\end{equation}

Applying Gronwall's inequality, we obtain
\[
\begin{aligned}
Y(t)
&\le
\frac{2E_-(u_0)}{1-\Theta}
\exp\left(
\frac{b(p-1)\Theta}{1-\Theta}
\int_0^t e^{-b(p-1)\tau}\,d\tau
\right)
\\
&=
\frac{2E_-(u_0)}{1-\Theta}
\exp\left[
\frac{\Theta}{1-\Theta}
\left(1-e^{-b(p-1)t}\right)
\right]
\\
&\le
\frac{2E_-(u_0)}{1-\Theta}
\exp\left(\frac{\Theta}{1-\Theta}\right).
\end{aligned}
\]
Therefore,
\begin{equation}
\label{equ:mass-critical-decay}
q(u(t))
\le
\frac{2E_-(u_0)}{1-\Theta}
\exp\left(\frac{\Theta}{1-\Theta}\right)e^{-2bt}.
\tag{B.15}
\end{equation}
Since $b\ge a>0$, it follows that
\[
q(u(t))
\le
\frac{2E_-(u_0)}{1-\Theta}
\exp\left(\frac{\Theta}{1-\Theta}\right)e^{-2at}.
\]
Hence
\[
\sup_{t\ge0}e^{at}q(u(t))^{1/2}<\infty.
\]

\end{proof}

\begin{remark}
	In the above, we have proved $f(t)$ is finite. Together with the mass dissipation estimate (cf. Lemma \ref{lem:mass}) and the blow-up alternative, $u$ is global.
\end{remark}

\section*{Acknowledgements}
L. Fanelli and Y. Wang are partially supported by the Basque Government through the BERC 2022--2025 program and by the research project PID2024-155550NB-100 funded by MICIU/AEI/10.13039/501100011033 and FEDER/EU.\,\,
L. Fanelli is also supported by the project IT1875-26 funded by the Basque Government.\,\,
Y. Wang is also supported by a Juan de la Cierva fellowship funded by MICIU/AEI/10.13039/501100011033, under Grant JDC2024-053285-I.\,\,
J. Zheng was supported by  National key R\&D program of China: 2021YFA10\\02500 and  NSFC Grant 12271051.\,\,H. Mizutani is partially supported by JSPS KAKENHI Grant Numbers JP21K03325, JP24K00529, JP26K22272 and JP26K00612.


\begin{thebibliography}{99}
	
	
	\bibitem{Bouclet}J.-M. Bouclet and N. Tzvetkov, \emph{Strichartz estimates for long range perturbations,} Amer. J. Math. 129  (2007), 1565-1609.
	
	\bibitem{Burq-Bouclet} J.-M. Bouclet, and N. Burq, \emph{Sharp resolvent and time-decay estimates for dispersive equations
		on asymptotically Euclidean backgrounds}. Duke Math. J. 170 (2021), 2575-2629.
	
	
	
	\bibitem{Fanelli2} N. Boussaid, P. D'Ancona and L. Fanelli, Virial identity and weak dispersion for the magnetic Dirac equation.
	J. Math. Pures Appl. (9) (2011), No. 2, 137-150.
	
	\bibitem{BGT} N. Burq, P. G\'erard, and N. Tzvetkov, Strichartz inequalities and the nonlinear Schr\"odinger equation on compact manifolds. Amer. J. Math., 126(3):569-605, 2004.
	
	
	\bibitem{Cacciafesta} F. Cacciafesta, Smoothing estimates for the variable coefficients Schr\"odinger equation with electromagnetic potentials. J. Math. Anal. Appl. 402 (2013), 286-296.
	
	\bibitem{CD16} B. Cassano, P. D'Ancona, Scattering in the energy space for the NLS with variable coefficients, Math. Ann., 366(2016), 479-543.
	
	\bibitem{Caz} T. Cazenave, Semilinear Schr\"odinger equations, volume 10 of Courant Lecture Notes in Mathematics.
	New York University, Courant Institute of Mathematical Sciences, New York; American Mathematical Society, Providence, RI, 2003.
	
	\bibitem{CCL} J. Colliander, M. Czubak and J. Lee, Interaction Morawetz estimate for the magnetic Schr\"odinger
	equation and applications. Adv. Differ. Equ. 19(2014), 805-832.
	
	
	
	\bibitem{CKSTT04} J. Colliander, M. Keel, G. Staffilani, H. Takaoka, and T. Tao. Global existence and scattering for rough solutions of a nonlinear Schr\"odinger equation on R3. Comm. Pure Appl. Math., 57(2004), 987-1014.
	
	\bibitem{Fanelli-Vega} L. Fanelli and L. Vega, Magnetic virial identities, weak dispersion and Strichartz inequalities, Math. Ann. 344 (2009), 249-278.
	
	\bibitem{DFVV-endpoint} P. D'Ancona, L. Fanelli, L. Vega and N. Visciglia, Endpoint Strichartz estimates for the magnetic Schr\"odinger equation. J. Funct. Anal. 258 (2010), 3227-3240.
	
	\bibitem{DiSj} M. Dimassi, J. Sj\"ostrand, Spectral asymptotics in semiclassical limit, London Mathematical Society, Lecture Notes Series, 268, Cambridge University Press, 1999.
	
	\bibitem{Fib01} G. Fibich, Self-focusing in the damped nonlinear Schr\"odinger equation, SIAM J. Appl. Math. 61 (2001),  1680-1705.
	
	\bibitem{GOV1994} J. Ginibre, T. Ozawa, G. Velo, {\it On the existence of the wave operators for a class of nonlinear Schr\"odinger equations}, Annales de l'I. H. P., Physique th\'eorique  \textbf{60} (1994), 211--239.
	
	\bibitem{GV85} J. Ginibre and G. Velo. Scattering theory in the energy space for a class of nonlinear Schr\"odinger equations. J. Math. Pures Appl. (9), 64(4):363-401, 1985.
	
	
	\bibitem{GRH80} M. V. Goldman, K. Rypdal, and B. Hafizi, Dimensionality and dissipation in Langmuir collapse, Phys. Fluids, 23 (1980), 945-955.
	
	\bibitem{Guo-Zhu}	
	Q. Guo and S. Zhu, Sharp criteria of scattering for the fractional NLS, Preprint, arixv: 1706. 02549.
	\bibitem{Saanouni}
	M. Hamouda, M. Majdoub and T. Saanouni, Damping effects on global existence and scattering for an inhomogeneous NLS equation with inverse-square potential, Z. Angew. Math. Phys. 76, 252 (2025).
	
\bibitem{Hassell-Zhang}
A. Hassell and  J. Zhang, 
Global-in-time Strichartz estimates on nontrapping, asymptotically conic manifolds.
Anal. PDE 9 (2016), no. 1, 151-192.

	\bibitem{Hormander} L. H\"ormander, The Analysis of Linear Partial Differential Operators III, Springer, Berlin, Heidelberg, New York, Tokyo, 1985.
	
	\bibitem{Inui} T. Inui. Asymptotic behavior of the nonlinear damped Schr\"odinger equation. Proc. Amer. Math. Soc., 147(2):763-773, 2019.
	
	\bibitem{KatoYajima} 
	T. Kato, K. Yajima, {\it Some examples of smooth operators and the associated smoothing effect}, Rev. Math. Phys. \textbf{1}(04) (1989), 481--496.
	
	\bibitem{KeelTao} M. Keel, T. Tao, {\it Endpoint Strichartz estimates}, Amer. J. Math. \textbf{120} (1998), 955--980. 
	
	\bibitem{KPV} C. E. Kenig, G. Ponce, L. Vega, {\it Oscillatory integrals and regularity of dispersive equations}. Indiana Univ. Math. J. \textbf{40} (1991), no. 1, 33--69.
	
	
	
	\bibitem{LS25} D. Lafontaine and B. Shakarov, Scattering for defocusing cubic NLS under locally damped strong trapping, Annales Henri Lebesgue, 9 (2026), 147-184.
	
	\bibitem{LS26} D. Lafontaine and B. Shakarov, Scattering for defocusing NLS with inhomogeneous nonlinear damping and nonlinear trapping potential, arxiv: 2603. 11765.
	
	\bibitem{Leinfelder}
	H. Leinfelder, C. Simader, Schr\"odinger operators with singular magnetic vector potentials, Math. Z. 176 (1981), 1-19.		

\bibitem{MMT-JFA}
J. Marzuola, J. Metcalfe and D. Tataru, Strichartz estimates and local smoothing estimates for asymptotically flat Schr\"odinger equations.
J. Funct. Anal. 255 (2008), no. 6, 1497-1553.
	
	\bibitem{Miz13} H. Mizutani, Strichartz estimates for Schr\"odinger equations with variable coefficients and unbounded potentials,  Analysis and PDE, 6(2013), 1857-1898.
	
	\bibitem{Staffilani-Tataru}
	G. Staffilani and D. Tataru, Strichartz estimates for a Schr\"odinger operator with nonsmooth coefficients, Comm. Partial Differential Equations {\bf 27} (2002), no.~7-8, 1337--1372.
	
	\bibitem{Stein} E. M. Stein, \emph{Harmonic analysis: real-variable methods, orthogonality, and oscillatory integrals}. Princeton Mathematical Series, vol.~43, Princeton University Press, Princeton, NJ, 1993, With the assistance of Timothy S. Murphy, Monographs in Harmonic Analysis, III.
	
	
	\bibitem{Tsu84} M. Tsutsumi, Nonexistence of global solutions to the Cauchy problem for the damped nonlinear Schr\"odinger equations, SIAM J. Math. Anal. 15 (1984), no. 2, 357-366.
	
	\bibitem{Weinstein}
	M.~I. Weinstein, Nonlinear Schr\"odinger equations and sharp interpolation estimates, Comm. Math. Phys. {\bf 87} (1982/83), no.~4, 567-576.
	
	\bibitem{YZ} K. Yajima, G. Zhang, Local smoothing property and Strichartz inequality for Schr\"odinger equations with potentials superquadratic at infinity, J. Differential Equations. \textbf{202} (2004), 81--110. 
	
	\bibitem{YNC} F. Yang, Z.-H. Ning, and L. Chen, Exponential stability of the nonlinear Schr\"odinger equation with locally distributed damping on compact riemannian manifold. Advances in Nonlinear Analysis, \textbf{10} (2021), 569--583
	583, 2021.
	
	\bibitem{Yudovich} V. I. Yudovich, Non-stationary flow of an ideal incompressible liquid. U.S.S.R. Comput.  Math. Phys., \textbf{3} (1967), 1407--1456.
	
\end{thebibliography}
\end{document}